\documentclass[12pt]{elsarticle}
\usepackage{amsmath}
\usepackage{amscd}
\usepackage{eucal}
\usepackage{amssymb}
\usepackage{epic}
\usepackage{graphicx}
\usepackage{color}
\usepackage{mathrsfs}
\usepackage{multirow}
\usepackage{subfigure}
\usepackage{amsthm}
\usepackage{amssymb}
\usepackage{amsbsy}
\usepackage{extarrows}
\usepackage{enumerate}
\usepackage{verbatim}
\usepackage[titletoc]{appendix}

\newtheorem{theorem}{Theorem}[section]
\newtheorem{remark}{Remark}[section]
\newtheorem{lemma}{Lemma}[section]
\newtheorem{definition}{Definition}[section]
\usepackage[only,llbracket,rrbracket,llparenthesis,rrparenthesis]{stmaryrd}
\usepackage{accsupp}

\numberwithin{equation}{section}
\numberwithin{table}{section}
\newcommand{\bV}{{\boldsymbol{V}}}
\usepackage{lineno,hyperref}
\modulolinenumbers[200]
\renewcommand{\div}{{\nabla \cdot}}
\newcommand{\eps}{\varepsilon}
\newcommand{\bb}{{\boldsymbol{b}}}
\newcommand{\bn}{{\boldsymbol{n}}}
\newcommand{\bq}{{\boldsymbol{q}}}
\newcommand{\bpsi}{{\boldsymbol{\psi}}}
\newcommand{\bone}{{\boldsymbol{I}}}
\newcommand{\TTh}{\mathscr{T}_h}
\newcommand{\pTh}{\partial\mathscr{T}_h}
\def\form(#1,#2){\left(#1,#2\right)}
\def\forme(#1,#2){\left<#1,#2\right>}
\newcommand{\norm}[1]{\ensuremath{\lVert{#1} \rVert}}
\def\jump#1{\left[\hskip -3.5pt\left[#1\right]\hskip -3.5pt\right]}
\def\normmm(#1){|||{#1}|||}
\newcommand{\wgrad}{\nabla_w}
\newcommand{\mwgrad}{\nabla_{w,k}}

\newcommand{\wdiv}{\nabla_{w}  \cdot}
\newcommand{\mwdiv}{\nabla_{w,k}  \cdot}
\newcommand{\eh}{e_h}
\newcommand{\uh}{u_h}

\newcommand{\uhb}{u_{hb}}
\newcommand{\uhz}{u_{h0}}
\newcommand{\vh}{v_h}
\newcommand{\vhb}{v_{hb}}
\newcommand{\vhz}{v_{h0}}
\newcommand{\wh}{w_h}
\newcommand{\whz}{w_{h0}}
\newcommand{\whb}{w_{hb}}

\newcommand{\Rh}{R_h}
\newcommand{\vz}{v_0}

\newcommand{\bv}{{\boldsymbol{v}}}
\newcommand{\bvz}{{\boldsymbol{v_0}}}
\newcommand{\bvb}{{\boldsymbol{v_b}}}

\newcommand{\half}{\frac{1}{2}}

\newcommand{\vb}{v_{b}}

\newcommand{\ehz}{e_{h0}}

\newcommand{\ah}{a_h}
\newcommand{\vt}{v_{\theta}}
\newcommand{\wt}{w_{\theta}}

\newcommand{\tp}{\tau^+}

\def\div{\nabla\cdot}%

\renewcommand{\div}{{\operatorname{div}}}

\DeclareMathAlphabet{\mathpzc}{OT1}{pzc}{m}{it}

\def\jump#1{\left[\hskip -3.5pt\left[#1\right]\hskip -3.5pt\right]}

\makeatletter

\newcommand{\Rmnum}[1]{\expandafter\@slowromancap\romannumeral #1@}
\def\form(#1,#2){\left(#1,#2\right)}
\def\forme(#1,#2){\left<#1,#2\right>}
\def\normmm(#1){|||{#1}|||}

\journal{}

\begin{document}

\begin{frontmatter}
\title 
{Robust a-posteriori error estimates for weak Galerkin method for the convection-diffusion problem }



%


\author[mymainaddress]{Natasha Sharma\corref{mycorrespondingauthor}}
\cortext[mycorrespondingauthor]{Corresponding author}
\ead{nssharma@utep.edu}

\address[mymainaddress]{Department of Mathematical Sciences, University of Texas of El Paso, El Paso, TX 79968}

\begin{abstract}
We present a robust a posteriori error estimator for the weak Galerkin finite element method applied to stationary convection-diffusion equations in the convection-dominated regime. The estimator provides 
global upper and lower bounds of the error 
 and is robust in the sense that upper and lower bounds are uniformly bounded with respect to the diffusion coefficient. Results of the numerical experiments are presented to illustrate the performance of the error estimator. 
\end{abstract}


\end{frontmatter}

\linenumbers
\section{Introduction}\label{sec:Intro}
We consider the following convection-diffusion equations as our model problem
\begin{subequations}\label{eq:cd}
\begin{align}
  - \div \big(\eps \nabla u\big) + \div\big( \bb  u  \big) + a u&= f \, \text{ in } \Omega ,  
\label{eq:cd1} \\
  u &= 0\, \, \text{ on } \Gamma, \label{eq:cd2}
\end{align}
\end{subequations}
 where $\Omega$ is a polygonal domain in $\mathbb{R}^2$ with boundary $\Gamma$ and where the data of the problem~\eqref{eq:cd} and the right-hand side in~\eqref{eq:cd1} satisfy the following assumptions:
\begin{itemize}
\item[A1.]{$\eps>0$. 
}
\item[A2.]{ $f \in L^2(\Omega)$, $\bb \in W^{1,\infty}(\Omega)^2$, $a \in L^{\infty}(\Omega)$.}
\item[A3.]{ There exists a positive constant $c_0$ 
such that
\begin{align}
0<c_0  \le \frac{1}{2} \div \bb + a \quad \text{on }\Omega.\label{convection_bd}
\end{align}
}
\end{itemize}
One of the challenges in the numerical approximations to~\eqref{eq:cd} is that when the problem is convection-dominated, the solutions to these problems possess layers of small width. In the presence of these layers, the solutions and/or their gradients change rapidly and as a consequence, standard finite element methods give inaccurate approximations unless the mesh size is fine enough to capture the layers.

In response to this challenge, several numerical approaches have been proposed over the years including stabilized methods~\cite{brooks1982streamline, hughes1979finite,brezzi1999priori, brezzi2000residual, burman2005stabilized}, discontinuous Galerkin (DG) methods~\cite{cockburn1999some, houston2002discontinuous, ayuso2009discontinuous, cockburn1999discontinuous, zarin2005interior, becker2000discontinuous, buffa2006analysis,schotzau2009robust, zhu2011robust, chen2016robust} and most recently, the weak Galerkin (WG) methods~\cite{CHEN2017107, LinYeZhaZhu18}. In particular, the WG methods use discontinuous approximations and have gained popularity owing to their attractive properties such as mass conservation, flexibility of geometry, flexibility on the choice of approximating functions~\cite{WangYe13}. It is interesting to note that over and above the advantages enjoyed by WG methods, the additional advantages of the WG scheme proposed in~\cite{LinYeZhaZhu18} is that it assumes a simple form and does not require any strict assumptions on the convection coefficient, a requirement which is crucial for several of the existing methods such as~\cite{ayuso2009discontinuous,schotzau2009robust, zhu2011robust, CHEN2017107}. Additionally, this simple WG scheme converges at the rate of $\mathcal{O}(h^{k+1/2})$ in the strongly advective regime, $k$ here denoting the polynomial order.

The motivation of our paper stems from the fact that despite the advantages of this simple WG scheme, it exhibits a poor performance in the intermediate regime see~\citep[Example 2]{LinYeZhaZhu18}, for instance. Our goal in this paper is to present a posteriori error analysis for the simple WG method proposed in~\cite{LinYeZhaZhu18}. We further demonstrate that by relying on adaptively refined meshes based on a posteriori residual-type estimator, we can retrieve the optimal order of convergence for all the regimes not just the convection-dominated regime. Additionally, we also prove that this estimator provides global upper and lower bounds of the error and is robust in the sense that upper and lower bounds are uniformly bounded with respect to the diffusion coefficient.

Existing literature for convection-diffusion problems solved using adaptive finite element methods based on a posteriori error estimates was initialized by Eriksson and Johnson in~\cite{eriksson1993adaptive} and enriched by several works authored and co-authored by Verf\"urth in~\cite{verfurth2005robust,verfurth1998posteriori,tobiska2015robust}. Within the DG framework, a posteriori estimates were developed by Sch\"otzau and Zhu in~\cite{schotzau2009robust,zhu2011robust}, and by Ern and co-authors in~\cite{ern2008posteriori,ern2010guaranteed}. In~\cite{chen2016robust}, a posteriori analysis using the hybridizable DG (HDG) method was presented and the error analysis relies on the extra conditions imposed on the convection coefficient. Furthermore, the proof of the local efficiency of the error estimator explicitly imposes a condition on mesh size to prove local efficiency see~\citep[Lemma 5.3]{chen2016robust}.

The WG method was originally introduced in~\cite{WangYe13} by Wang and Ye as a class of finite element methods that employ weakly defined differential operators to discretize partial differential equations without demanding any fine tuning of parameters in its weak formulation. While much attention has been paid to developing weak Galerkin methods for a wide class of problems~\cite{LiWang13, WangYe16, MuWangYe14b, MuWangYeZhang15, MuWangYe15,MuWangWangYe13, MuWangYe14a, WangWang14, WangWang15}, there is not much development in the direction of adaptive WG methods and most of the existing a posteriori error analysis is derived for either the second-order elliptic equations~\cite{Chen2014, li2019posteriori,ZhangTieChen18,AdlerHuYe2018} or for the Stokes problem~\cite{zheng2017posteriori}.

The paper is organized as follows. 
Section \ref{sec:wg} develops the weak Galerkin finite element method for the model problem. Section \ref{sec:posteriori} derivation of the a posteriori error analysis is presented. Section \ref{sec:numerical} demonstrates the effectiveness of our method  with results of  numerical experiments.
\section{Weak Galerkin Formulation}\label{sec:wg}
 Let $\{\TTh\}_{h>0}$ be a uniformly shape-regular partition of
 $\Omega$ into rectangular or triangular cells. For any $T \in \TTh$, let  $h_T$ denote the diameter of $T$, $h=\max_{T\in \TTh}h_T$, the mesh size of $\TTh$ and $\partial T$ denote the boundary of T. 

 Let $\mathscr{E}_h$ and $\bar{\mathscr{E}_h}$ denote the collection of all the interior edges and all the edges associated with the triangulation $\TTh$ respectively.
  Also, for any $E\in\bar{\mathscr{E}_h}$, let $h_E$ denote the length of $E$.
   We assume that there exists a constant $\kappa>0$ such that for each $T\in\TTh$, we have
   \begin{align}\label{shape-reg}
   h_T \le \kappa h_E
   \end{align}

 We also introduce the shorthand notation for broken inner products over
 $\TTh$ and $\pTh$ respectively as
 \begin{align*}
   \form(f,g)_{\TTh} &:= \sum_{T\in\TTh}\form(f,g)_T= \sum_{T\in\TTh}\int\limits_{T} f g\,  dx,  \\
   \forme(f,g)_{\pTh} &:= \sum_{T\in\TTh}\forme(f,g)_{\partial T}
  = \sum_{T\in\TTh} \ \int\limits_{\partial T} f \ g dx.
 \end{align*}
 In the above, we have used the shorthand $\partial T$ to denote the boundary of each cell $T\in \TTh$.
 The notation $\norm{\cdot}_S$ denotes the $L^2$ norm over any $S$ belonging to $\TTh$ or $\mathscr{E}_h$.       
  Here and in the sequel, we employ the standard notation for well-known Lebesque and Sobolev spaces and norms defined on them (cf., e.g,~\cite[Section 1.2]{ciarlet2002finite}). Throughout this paper, we use the symbol $\lesssim$ to denote bounds involving positive constants independent of the local mesh size and $\eps$.
  
 \subsection{Weak Differential Operators}
 Following~\cite{WangYe13}, we introduce the definitions of the weak gradient and divergence operators. For any $T\in \TTh$, we 
              let $V(T)$, denote the space of weak functions on T as
                  \begin{align*}
                      V (T)&:=\Big\{ v = \{ \vz, \vb\}: \vz \in L^2(T), \vb \in H^{\half}(\partial T),                       \Big\}
                         \end{align*}
                      and define the weak gradient as follows.
                  \begin{definition}[Weak Gradient] 
                  Given $ v\in V(T),$\\
                  $$(\wgrad v, \bq)_T:=- \int\limits_T\vz \div \bq \ dx + \int\limits_{\partial T}\vb \bn \cdot \bq \ ds \quad \forall \bq \in H(div;T),$$
                  where $\bn$ is the outward unit normal vector to $\partial T$ and 
                  \[
                  H(div;T):=\{\bq : \div \bq \in L^2(T)
                  \}.          
                  \]
                  \end{definition}

  On $T$, let $P_k(T)$ denote the space of all polynomials with degree no greater than $k$. For a given integer $k \ge 1$, let $V_{h,k}$ be the weak Galerkin finite element space corresponding to $\TTh$ defined as follows:
  \begin{align*}
  V_{h,k}:\{v=\{v_0,v_b\} \ : \ v_0|_T \in P_k(T), \ v_b|_{\partial T} \in P_k(\partial T), \ T\in\TTh\}, 
  \end{align*}
  and let $V_h^0$ be its subspace defined as:
   \begin{align*}
   V_{h,k
   }^0:\{v=\{v_0,v_b\}\in V_{h,k} \ : \ \ v_b=0, \text{ on } \partial \Omega\}. 
   \end{align*}
   
   We introduce the definition of the discrete weak gradient as described in~\cite{LinYeZhaZhu18}.
   \begin{definition}[Discrete Weak Gradient]
   For any $v=\{v_0,v_b\} \in V_{h,k}$ and for any $T \in \TTh$, 
        the discrete weak gradient $\mwgrad v \in  P_{k-1}(T)^2$ is defined on $T$
      as the unique polynomial satisfying
       \begin{equation}\label{def:wgrad}
         (\mwgrad v, \bpsi)_T = -(\vz, \div \bpsi)_T + \langle \vb, \bpsi \cdot \bn \rangle_{\partial T} \quad
         \forall \bpsi \in P_{k-1}(T)^2,
         \end{equation}
         where $\bn$ is the outward unit normal to $\partial T$.
   \end{definition}
We note that since the right-hand side of~\eqref{def:wgrad} defines a bounded linear functional on $H(div;T)$, 
        the definition of discrete weak gradient~\eqref{def:wgrad} is well defined for $\bpsi \in  H(div;T)$ as well.
       Thus, 
       \begin{align}\label{def:wgrad-discrete}
            (\mwgrad v, \bpsi)_T &= -(\vz, \div \bpsi)_T + \langle \vb, \bpsi \cdot \bn \rangle_{\partial T} \nonumber \\
            &=(\wgrad v, \bpsi)_T, \quad \bpsi \in H(div;T).     
        \end{align}

   Next, we introduce the weak divergence operator by first defining the space of weak vector-valued functions on $T$ as 
              \[
              \bV(T) := \Big\{  \{ \bvz, \bvb\}: \bvz \in L^2(T)^2, \bn \cdot\bvb \in H^{-\half}(\partial T) \Big\},
              \]
     so that the weak divergence operator can be defined as follows.
            \begin{definition}[Weak Divergence] For $\bv\in \bV(T)$,
              \[
              (\wdiv \bv, \psi)_T:= -(\bvz, \nabla \psi)_T + \langle \bvb \cdot \bn , \psi \rangle_{\partial T} \quad \psi \in H^1(T)
              \]
              where
              \[H^1(T):=\{v : \nabla v \in L^2(T)\} .\]
            \end{definition}
   Following~\cite{LinYeZhaZhu18}, the discrete weak divergence is defined as follows. 
   \begin{definition}[Discrete Weak Divergence]
      Given $v=\{\vz,\vb\} \in V_{h,k}$ and for any $T \in \TTh$, a discrete weak divergence $\mwdiv (\bb v) \in P_k(T) $ associated with $\bb$ can be defined as the unique polynomial satisfying 
            \begin{equation}\label{def:wdiv}
            (\mwdiv (\bb v), \psi)_T = -(\bb \vz, \nabla \psi)_T + \langle \bb \cdot \bn \vb , \psi \rangle_{\partial T} \quad
            \forall \psi \in P_k(T),
            \end{equation}
             where $\bn$ is the outward unit normal to $\partial T$.
                        
   \end{definition}  
 Notice that the above definition~\ref{def:wdiv} holds true for $\bpsi \in  H^1(T)$. This is because the right-hand side of~\eqref{def:wdiv} defines a bounded linear functional on $H^1(T)$.
         Thus for $ \psi \in H^1(T)$, we have
              \begin{equation}\label{def:wdiv-discrete}
                  (\mwdiv (\bb v), \psi)_T = -(\bb \vz, \nabla \psi)_T + \langle \bb \cdot \bn \vb , \psi \rangle_{\partial T}
                   =(\wdiv (\bb v), \psi)_T.      
                   \end{equation}   

In view of~\eqref{def:wgrad-discrete} and~\eqref{def:wdiv-discrete}, we drop the subscript ``$k $" from the discrete weak differential operator $\mwgrad$ and instead adopt the notation $\wgrad$ and $\wdiv$ whenever the context is clear.
                          
     \subsection{Weak Galerkin Method}                        For $\wh=\{\whz,\whb\}$ and $\vh=\{\vhz,\vhb\}$ in $V_{h,k}$, we define a bilinear form as 
                  \begin{align}\label{def:bilinear-form}
                    \ah(\wh,\vh):=&\eps (\wgrad \wh, \wgrad \vh)_{\TTh} + ( a \whz+ \wdiv (\bb \wh), \vhz)_{\TTh}\nonumber
                       \\
                           &
                           +
                           s(\wh,\vh),
                  \end{align}
                     where $  s(\wh,\vh)=\forme(\tp(\whz-\whb), \vhz -\vhb)_{\partial\TTh}$ with
                       $\tau^{+}$ defining the stabilization parameter defined as:
                       \begin{align*}
                                       \tau^{+}|_{\partial T}&=
                                       (\bb \cdot \bn) \ {{\bone}}_{{\partial T}^{+}} + \eps (\kappa h_T)^{-1} + (\norm{\bb}_{\infty}\eps^{-1} + 1)h_T 
                                       + \eps h_T^{-1},  \quad \forall T\in \TTh,    
                                       \end{align*}
                              with the constant $\kappa>0$ introduced in~\eqref{shape-reg} and
                                       \begin{align*}
                                       &I_{{\partial T}^+}(x)=\begin{cases}
                                       1 &\mbox{if } \bb(x) \cdot \bn(x)\ge 0,\\
                                       0 &\mbox{otherwise.} 
                                       \end{cases}
                                          \end{align*}
                                       A weak Galerkin approximation for ~\eqref{eq:cd} amounts to  
                                                      seeking $\uh = \{ \uhz, \uhb \} \in V_{h,k}^0$ satisfying the following equation:
                                                                       \begin{align}\label{eq:cd-wg}               
                                                                       a_h(u_h,v_h)= (f, \vhz) \quad &\forall v_h=\{\vhz,\vhb\} \in V_h^0.
                                                                       \end{align}

                                       We equip the space $V_h^0$ with the following weak Galerkin energy norm                           
                                              \begin{align}\label{eq:energy-norm}
                                              \normmm(v_h)_{}^2&:= \eps\norm{\mwgrad v_h}_{\TTh}^2
                                              + \norm{\vz}_{\TTh}^2
                                               + 
                                              \forme(\tau(\vhz-\vhb), \vhz -\vhb)_{\partial\TTh}.
                                              \end{align}
                                              where
                                               \begin{align*}  
                                              \tau|_{\partial T} =|\bb \cdot \bn|
                                              +
                          \big(\eps h_T^{-1} + \eps (\kappa h_T)^{-1} + 
                          \norm{\bb}_{\infty}\eps^{-1} + 1)h_T 
                          \big)                    
                                                \quad \forall T\in \partial\TTh.
                                              \end{align*}

         The present work modifies the WG method of Lin and co-authors~\citep[eq. (2.8)]{LinYeZhaZhu18} for solving the 
         convection-diffusion equation with an additional diffusion-dependent stability term $\forme((\eps \kappa h_T^{-1} + h_T)(\whz-\whb), \vhz -\vhb)_{\partial\TTh}$. The presence of this additional stability term enables us to bound several terms of the a posteriori error estimator.
         
          Associated with the solution $u$ satisfying~\eqref{eq:cd1}--\eqref{eq:cd2} and $u_h \in V_h$  satisfying~\eqref{eq:cd-wg}, we introduce the following weak function measuring the error
                                        \begin{align}
                                             e_h:=\{e_{h0}, e_{hb}\}= \{u_0 - u_{h0}, u_b - u_{hb}\},\label{err_rep}
                                             \end{align}
                where $u_0$ and $u_b$ are understood to be the restrictions of $u$ to the interior and the boundary of each $T\in \TTh$.                                            
            Furthermore, following the discussion in~\citep[Section 3]{WangYe13}, for any $T\in\TTh$, an application of Green formula reveals 
                                                \begin{align*}
                               (\nabla u, \bq)_T &= -(u,\div \bq)_T + \langle u,\bn \cdot \bq\rangle_{\partial T}\\
                                              &= -(u_0,\div \bq)_T + \langle u_b,\bn \cdot \bq\rangle_{\partial T}=(\wgrad u,\bq)_T, \quad \bq \in P_{m-1}(T)^2.
                                                \end{align*}                                    
                                                                                       
                             This allows us to measure the error in the energy norm:
                                    \begin{align}
                                    \normmm(u -\uh)^2 = \eps \norm{\wgrad u - \wgrad \uh}_{\TTh}^2 +
                                   \norm{\tau^{1/2}(\uhz-\uhb )}_{\partial\TTh}^2 + \norm{u-\uhz}_{\TTh}^2
                                   \label{def:norm_extension}
                                    \end{align}
  where $\wgrad \uh$ is understood to be the discrete weak gradient $\mwgrad u_h$ of $u_h\in V_{h,k}$ while $\wgrad u$ is the classical derivative of $u$.            

        The well-posedness of the WG formulation is guaranteed thanks to the following Lemma which appeared in of~\cite[Section 3.1]{LinYeZhaZhu18} and readily applies to our WG form.
       
          \begin{lemma}[Coercivity and Continuity]\label{lemma:well-posedness}
                     There exists constants $c_{1}$ and $c_{2}$ such that there holds
                     \begin{align}
                     \ah(w,v) &\le c_1 \normmm(w) \ \normmm(v) \quad \forall w, \ v \in V_h, \label{eq:cty}\\
                     c_2 \normmm(v)^2 &\le \ah(v,v),         \quad \forall  v \in V_h. \label{eq:coer}
                     \end{align}
                     \end{lemma}
                  \begin{proof}
                  The proof has been provided in subsection 3.1 of~\cite{LinYeZhaZhu18}.
                  \end{proof}  
                  
                  \begin{remark}\label{rem:motivation}
                  While the above Lemma~\ref{lemma:well-posedness} guarantees the well-posedness of the WG method~\eqref{eq:cd-wg}, the constant $c_1$ appearing in~\eqref{eq:cty} have an undesirable dependence on $\eps^{-1}$. This dependence arises when trying to control the term $(\wdiv(\bb u),w)$. Hence, our a posteriori error analysis cannot rely on the continuity property to control this term.
                   Instead, inspired by the strategy in~\cite{schotzau2009robust}, we control this term by
                including the following semi-norm $|\cdot|_*$ to our energy norm~\eqref{eq:energy-norm}.
                  \end{remark}

\begin{definition}[Operator] For any $\bq \in L^2(\Omega)^2$, we define the following semi-norm $|\cdot|_*$ as
\begin{align}\label{def:operator-norm}
|\bq|_*= \sup_{w\in H_0^1(\Omega)\setminus \{0\}}\frac{ \int_{\Omega}\bq \cdot \nabla w \ dx}{\normmm(w)}.
\end{align}
The definition above is well defined because $\bq$ admits the following Helmholtz decomposition
\begin{align}\label{eq:decomp-00}
\bq=\nabla \phi + \bq_0,
\end{align}
where $\phi \in H_0^1(\Omega)$ uniquely solves
\[
\int_{\Omega}\nabla \phi \cdot \nabla w \ dx = \int_{\Omega}\bq \cdot \nabla w \ dx \quad \forall w \in H_0^1(\Omega),
\]
and $\bq_0=\bq - \nabla \phi$ satisfies the divergence-free property 
\[
 \int_{\Omega}\bq_0 \cdot \nabla w \ dx = 0 \quad \forall w\in H_0^1(\Omega).
 \]
 The decomposition~\eqref{eq:decomp-00} is unique and orthogonal in $L^2(\Omega)^2$.
\end{definition}
Consequently, we have the following upper bound:\\
For $u_h=\{\uhz,\uhb\}\in V_{h,k}, \ v\in H^1_0(\Omega)$,  
  \begin{align}\label{eq:div-bd}
(\wdiv(\bb u_h),v)&\le |\bb \uhz|_* \normmm(v),
  \end{align}     
  where 
           $  \normmm(v)_{}^2 := \eps\norm{\nabla v}_{}^2 + \norm{v}_{}^2 $. Inequality~\eqref{eq:div-bd} holds because 
           \[
           (\wdiv(\bb \uh),v)_{\TTh}=-(\bb \uhz, \nabla v)_{\TTh}+\langle \bn \cdot \bb \uhb,v\rangle_{\pTh}
           \]
            and using the fact that $\langle  \bn \cdot \bb \uhb,v\rangle_{\pTh}=0$.\\ 
     \begin{remark}[Practical estimation of $|\bb (u-\uhz)|_*$]\label{rem:bound}
     Following the arguments presented in~\citep[Remark 3.5]{schotzau2009robust}, we use the following inequality 
     \[
     |\b(u-\uhz)|_* \lesssim \frac{1}{\sqrt{\eps}}\norm{u-\uhz}
     \]
     to serve as a computable upper bound for $|\bb (u -\uhz)|_*$.
  
     \end{remark}

   We close this section by proving the inf-sup condition for $\ah(\cdot,\cdot)$ on $H_0^1(\Omega)$. The proof follows the arguments presented in~\cite{schotzau2009robust}. However, since our definition of the energy norm $\normmm(u)+ |\bb u|_*$ differs from the one used in~\cite{schotzau2009robust}, we include the proof below for completeness.
                     \begin{theorem}[inf-sup condition] There exists a constant $C>0$ satisfying
                     \begin{align}\label{eq:inf-sup}
                     \inf_{u\in H_0^1(\Omega)\setminus\{0\}}\sup_{v\in H_0^1(\Omega)\setminus\{0\}}\frac{\ah(u,v)}{(\normmm(u)+|\bb u|_*)\normmm(v)}
                     \ge C>0.
                     \end{align}
                     \end{theorem} 
                     \begin{proof}
                     Let $u\in H_0^1(\Omega)$ and $0<\theta<1$. There exists $w_{\theta}\in H_0^1(\Omega)$ such that
                     \begin{align*}
                     \normmm(w_{\theta})=1, \quad (\wdiv(\bb u),w_{\theta}) \ge \ \theta |\bb u|_*.
                     \end{align*}
                     \begin{align*}
                     (\eps \nabla u,-\nabla w_{\theta} ) + (a u, -w_{\theta}) 
                     &\le \norm{\eps^{1/2}\nabla u} \
                     \norm{\eps^{1/2}\nabla w_{\theta}} + \| a\|_{\infty} \norm{u} \ \norm{w_{\theta}}\\
                     &\le \max\{1,\| a\|_{\infty} \} \ \normmm(u).
                     \end{align*}           
                     Thus,
                     \begin{align}\label{eq:lb1}
                     \ah(u,w_{\theta})&=  (\eps \nabla u,\nabla w_{\theta} ) + (a u, w_{\theta})
                       +(\wdiv(\bb u),w_{\theta}) \nonumber \\
                       &\ge-c_1 \ \normmm(u) \ + \theta |\bb u|_*,
                        \end{align}
                       where $c_1=\max\{1,\| a\|_{\infty} \}$. 
                       Noting that
                                          
                       Next, we define $\vt= u+\alpha \wt  $ where $\alpha>0$ is chosen suitably.
                       Using 
                        \begin{align}\label{eq:lb0}
                                                      \ah(u,u)&\ge \eps\norm{\wgrad u}^2+(c_0u,u) \nonumber\\
                                                      &\ge c_*\normmm(u)^2 \quad \text{where } c_*=\min\{1,c_0\}
                                                       \end{align}
                        and \eqref{eq:lb1}, we have
                       \begin{align}\label{eq:lb2}
                       \ah(u,\vt)&=  \ah(u,u) + \alpha\ah(u,\wt)
                       \nonumber\\
                       &\ge   c_*\normmm(u)^2 + \alpha \big(-c_1 \ \normmm(u) \ + \theta |\bb u|_*\big).
                       \end{align}
                    We pick $\alpha= c_*\normmm(u)/(c_1+1)$ so that,
                    \begin{align}
                    \ah(u,\vt)&\ge(c_1+1)^{-1}c_*\big(\normmm(u) +  \theta |\bb u|_*\big)\normmm(u) \quad \text {and,}
                    \nonumber
                    \\
                    \normmm(\vt)&=(c_1+1)^{-1}\{c_1+1+c_*\}\normmm(u)\quad 
                    \nonumber
                     \\
                     \text{ thus,}\quad
                     \frac{\ah(u,\vt)}{\normmm(\vt)}
                     &\ge 
                     \frac{c_*\big(\normmm(u) +  \theta |\bb u|_*\big)}{\{c_1+1+c_*\}}=C\big(\normmm(u) +  \theta |\bb u|_*\big)\nonumber
                    \end{align}   
                    where $C=c_*\{c_1+1+c_*\}^{-1}>0$.
                    
                    Consequently, for any $u\in H_0^1(\Omega), \ \theta\in(0,1)$, we have
                    \begin{align*}
                    C\big(\normmm(u) +  \theta |\bb u|_*\big) &\le
                     \frac{\ah(u,\vt)}{\normmm(\vt)} \le  \sup_{v\in H_0^1(\Omega)}\frac{\ah(u,v)}{(\normmm(u)+|\bb u|_*)\normmm(v)}
                    \end{align*}
                    from which, the result follows.
                     \end{proof}           
\section{A Posteriori Error Analysis} \label{sec:posteriori}
    In this section, we present a residual-based estimator and prove its reliability and efficiency. 
                        \begin{definition}[Estimator] We introduce the estimator in terms of
                               the cell and edge indicators defined as shown below:
                                 \begin{align}
                                 \eta_h^2:=\sum\limits_{T \in \TTh}\Big(\eta_{T,1}^2 + 
                                 \eta_{T,2}^2\Big)
                                  +                       
                                 \sum\limits_{E \in \mathscr{E}_h}\eta_{E}^2,\label{eq:estimator}
                                 \end{align}
                                 where the cell and edge residuals are 
                                  \begin{align}\label{def:estimator}
                                &\eta_{T,1}^2:= \alpha_T^2\norm{R_h}^2_{T}, \quad  \text{with} \ \Rh :=f_h+ \div (\eps \wgrad \uh) -\wdiv(\bb\ u_{h0}) - a_h \uhz,
                                 \nonumber\\
                                  &\eta_{T,2}^2:= \langle\tau(\uhz-\uhb),\uhz-\uhb\rangle_{\partial T},
                                \nonumber\\
                                &\eta_{E}^2:=\alpha_E \eps^{-1/2}\norm{J_h}_{E}^2,   \text{ with } J_h:=\jump{\eps\bn \cdot \wgrad u_h},
                                \end{align}
                               and the weights are according to~\citep[equation (3.4)]{verfurth2005robust}
                               \begin{align}\label{est:weights}
                               \alpha_T=\min\{h_T\eps^{-1/2},c_0^{-1/2}\},\ \text{ and } \ \alpha_E=\min\{h_E\eps^{-1/2},c_0^{-1/2}\}.
                               \end{align}
 We note that if $c_0 =0$, we define the weights as follows                              
    \begin{align}\label{est:weights2}
                                   \alpha_T=\min\{h_T\eps^{-1/2},1\},\ \text{ and } \ \alpha_E=\min\{h_E\eps^{-1/2},1\}.
                                   \end{align}                             

                                \end{definition}
  We also define the oscillation of data as
  \begin{align}
  \label{def:osc}
  osc(f, a)=\Big(\sum_{T\in\TTh}\alpha_T^2\big( \norm{f-f_h}_T^2 +\norm{(a-a_h)\uhz}_T^2\big)\Big)^{1/2}.
  \end{align}
                                
%
         \subsubsection{Reliability of Estimator}
         This subsection is devoted to proving the reliability for the estimator defined in~\eqref{eq:estimator}. The proof of the reliability relies on decomposing the discretization error into its conforming and non-conforming component and deriving upper bounds for each component. To this end, we first introduce a suitable conforming discrete space 
                                             $V_{h}^c= V_h\cap H_0^1(\Omega)$.
                                              
                    For any $v_h=\{\vhz,\vhb\}\in V_h$  we associate a conforming approximation $\vhz^c\in V_h^c$. Construction of such an approximation is a standard DG tool in the error analysis (see~\citep[Theorems 2.2 and 2.3]{karakashian2003posteriori} for instance). 
  This allows us to decompose $v_h $ as
                                                               \begin{align}\label{eq:decomp}
                                                               v_h=v_h^c + v_h^r
                                                               \end{align}
 where the nonconforming component $v_h^r=\{\vhz-v_h^c,\vhb-v_h^c\} \in V_h$. \\                      
 The above conforming approximation satisfies the following properties:
 \begin{subequations}\label{eq:dg-approx}
    \begin{align}
                                                   \norm{\vhz - v_h^c}_{\TTh}^2 \lesssim \sum_{E\in\mathscr{E}_h}\int_E h_E \jump{\vhz}^2 \ ds,\\
                                                    \norm{\nabla \vhz - \nabla v_h^c}_{\TTh}^2 \lesssim \sum_{E\in\mathscr{E}_h}\int_E h_E^{-1} \jump{\vhz}^2 \ ds .
             \end{align}      
 \end{subequations}
                                                
 Thanks to the single-valuedness of $\vhb$ over each edge $E\in \mathscr{E}_h$, the jump term in the rhs of~\eqref{eq:approx0} and~\eqref{eq:approx1} can be expressed as
                                                           \begin{align}
                                                            \sum_{E\in\mathscr{E}_h}\int_E \jump{\vhz}^2 \ ds
                                                           &=\sum_{E\in\mathscr{E}_h}\int_E  \jump{\vhz-\vhb}^2 \ ds
                                                           \nonumber
                                                           \\
                                                            &\le \sum_{T\in\mathscr{T}_h}\int_{\partial T} 2 (\vhz-\vhb)^2 \ ds .
                                                            \label{eq:jmp-bd}
                                                           \end{align}
Using~\eqref{eq:jmp-bd} in~\eqref{eq:dg-approx} and by using $\kappa h_T\le h_E,$ and $ h_E\le h_T$, we have                                                           
   \begin{align}
                                                 \norm{\vhz - v_h^c}_{\TTh}^2 &\lesssim  
                                                  \sum_{T\in\mathscr{T}_h}h_T\int_{\partial T}  (\vhz-\vhb)^2 \ ds
                                                 ,\label{eq:approx0}\\
        \norm{\nabla \vhz - \nabla v_h^c}_{\TTh}^2 &\lesssim    \sum_{T\in\mathscr{T}_h}(\kappa h_T)^{-1}\int_{\partial T}  (\vhz-\vhb)^2 \ ds \label{eq:approx1}.
                                                 \end{align}                           
The lemma below provides bounds for the non-conforming component.                                                           
     \begin{lemma}[Nonconforming term bound]
      \label{lemma-nconf}
                                                  For $\vh\in V_h$ admitting the decomposition~\eqref{eq:decomp} the following hold true:    
                                                 \begin{align}
                                                 \ah(\vh^r,w)
                                          &\lesssim 
                                          \max\{1,\norm{a}_{\infty},\norm{\bb}_{\infty} \} \
                                          \normmm(\vh^r)
                                         \ \normmm(w), \quad w\in H_0^1(\Omega),
                                          \label{eq:bound-non-conform} \\                                         
                                              \normmm(\vh^r)+|\bb \vhz^r|_*
             & \lesssim                                                                                                                             \sum_{T\in\TTh} (\eps h_T^{-1} + \eps (\kappa h_T)^{-1} + h_T)\norm{\vhz-\vhb}^2_{\partial T}.    \label{eq:bound-non-conform1}  
            \end{align}

                                                   \end{lemma}                                               
\begin{proof} To see~\eqref{eq:bound-non-conform}, we apply the definitions of the WG form and weak divergence    
  \begin{align*}
                                                 \ah(\vh^r,w)&=(\eps \wgrad \vh^r, \nabla w)_{\TTh} - (\bb \vhz^r, \nabla w)_{\TTh} + 
                                                       (a \vhz^r,  w)_{\TTh} + 
                                                              \langle \bn\cdot \bb \vhb^r, w\rangle_{\pTh}\nonumber\\
 &\le \Big(\normmm(\vh^r) + \norm{\bb}_{\infty} \eps^{-1/2}\norm{\vhz^r}
                                            + \norm{a}_{\infty} \norm{\vhz^r}
                                           \Big)\normmm(w),\\                                                             
 &\le\max\{1,\norm{a}_{\infty},\norm{\bb}_{\infty} \}\normmm(\vh^r) \ \normmm(w)
 , \quad w\in H_0^1(\Omega).
   \end{align*} 
 Here we have used~\eqref{eq:approx0} and $  \langle \bn\cdot \bb \uhb^r, w\rangle_{\pTh}=0$.                                                  
                 
Next, for deriving~\eqref{eq:bound-non-conform1}, we use the definition~\eqref{def:wgrad} of the weak gradient in conjunction with the trace, inverse and Cauchy-Schwarz inequalities and~\eqref{eq:approx0} to obtain                                              
                                         \begin{align}
                                                 \norm{\wgrad \vh^r}_T^2 &=(\nabla \vhz^r, \wgrad \vh^r)_T -\langle \vhz^r-\vb^r, \bn \cdot \wgrad \vh^r\rangle_{\partial T}
                                                 \nonumber
                                                 \\
                                                 &\le \Big( \norm{\nabla \vhz^r}_T + Ch_T^{-1/2}\norm{\vhz^r-\vhb^r}_{\partial T}\Big) \  \norm{\wgrad \vh^r}_{T}
                                                 \nonumber
                                                 \\        
                                                          \eps^{1/2} \norm{\wgrad \vh^r}_{T} 
                                                          &\le \min\{1, C\}
                                                            \Big(  \eps\norm{\nabla \vhz^r}^2_{T} +   \eps h_T^{-1}\norm{\vhz-\vhb}^2_{\partial T}\Big)^{1/2} 
                                                            \nonumber
                                                            \\
    \text{Thus, summing }& \text{over all } T\in \TTh, \nonumber \\                                                        
                                                     \eps\norm{\wgrad \vh^r}_{\TTh}^2            
                                                     &\lesssim 
                                                    \sum_{E\in\mathscr{E}_h}\int_E \eps h_E^{-1} \jump{\vhz}^2 \ ds +    \sum_{T\in\TTh} \eps h_T^{-1}\norm{\vhz-\vhb}^2_{\partial T}.\label{eq:bd0}
                                                           \end{align}
Using~\eqref{eq:jmp-bd} in~\eqref{eq:bd0} and~\eqref{eq:approx0} we obtain
                                                \begin{align}\label{eq:norm-bd0}
                                                   \eps\norm{\wgrad \vh^r}_{\TTh}^2  +\norm{\vhz^r}^2_{\TTh}          
                                                              \lesssim 
                                                               \sum_{T\in\TTh} 
                 (\eps h_T^{-1} + \eps (\kappa h_T)^{-1} + h_T)\norm{\vhz-\vhb}^2_{\partial T}.
                                                \end{align}
                                                
             Finally, for bounding $|\bb \vhz^r|_*$, we use~\eqref{eq:approx0}, Cauchy-Schwarz inequality and the definition of $|\bb\vhz^r|_*$
                                                \begin{align}\label{eq:norm-bd1}
                                               |\bb \vhz^r|_* &\lesssim  \eps^{-1/2}\norm{\bb}_{\infty}\norm{\vhz^r}_{\TTh}
                                               \nonumber
                                               \\
                                               &\lesssim  \norm{\bb}_{\infty}
                                                 \Big( \sum_{T\in\TTh}  \eps^{-1} h_T\norm{\vhz-\vhb}^2_{\partial T}\Big)^{1/2}.
                                                \end{align}
   \end{proof}    
   Next, we proceed to control the conforming term by first deriving the error equation. To this end, we introduce the following continuous subspace
                                                \[
                                                V_{h,1}^c:=\{ w \in H_0^1(\Omega):w|_K \in P_1(T), \ T\in\TTh \}
                                                \]
                                                 and observe that by setting $v_b$ as trace of $v$ on all the edges $E \in \pTh$, 
                                                 $V_{h,1}^c$ can be naturally embedded in $V_h$. This approach was used within the adaptive WG framework by Chen et al.~\cite{Chen2014} for obtaining partial orthogonality for second order elliptic problems.
                                                 
           We also introduce the following Cl\'ement interpolation operator specially designed for convection diffusion problems by Verf\"urth~\citep[Lemma 3.3]{verfurth2005robust} and references therein.
                                            We denote this operator by $\Pi_h: H_0^1\rightarrow V_{h,1}^c$ satisfying $\normmm(\Pi_h v)\lesssim \normmm(v)$ 
                                             and
                                             
                                                   \begin{align}                                       
                                              \Big( \sum_{T\in \TTh}   \alpha_T^{-2} \norm{v - \Pi_h v}^2_{T}\Big)^{1/2}
                                               &\lesssim  \normmm(v),\label{eq:inter_cell}\\
                                      \Big( \sum_{E\in \mathscr{E}_h}    \eps^{1/2}\alpha_E^{-1} \norm{v - \Pi_h v}_{E}^2\Big)^{1/2}&\lesssim \normmm(v), \quad v\in H_0^1(\Omega).\label{eq:inter_edge}
                                                   \end{align}                                                              
                              It is easy to see that for $u$ satisfying~\eqref{eq:cd} and $\uh \in V_h$ satisfying the  WG form~\eqref{eq:cd-wg}, the following holds true:
                                         \begin{align}\label{eq:partial-orth}
                                         \ah(u-\uh,\vh^c)=0 \quad  \ \forall v_h^c\in V_{h,1}^c.
                                         \end{align}            
                          As a consequence, we have the following error equation.                  
                             \begin{lemma}[Error Equation]
                             Let $u$ solve~\eqref{eq:cd1}--\eqref{eq:cd2}, $u_h \in V_{h}$ solve~\eqref{eq:cd-wg} decomposed as $\uh=u_h^c + \uh^r$ according to the decomposition~\eqref{eq:decomp}.
                              Then, for any $v \in H_0^1(\Omega)$, 
                                \begin{align}\label{eq:err-eqn}
                                     \ah(u-\uh,v)
                                     &=(osc(f,a),\hat{v})+(R_h, \hat{v})-\langle J_h, \hat{v}\rangle
                                       \end{align}
where $\hat{v}=v-\Pi^c_hv$ and $osc(f,a)$, $R_h$ and $J_h$ are defined in~\eqref{def:osc}~\eqref{def:estimator} respectively. 
\begin{proof} Since the exact solution $u$ solves~\eqref{eq:cd1}--\eqref{eq:cd2}, we obtain, by applying~\eqref{eq:partial-orth} for $\Pi^c_hv\in V_{h,1}^c$,
   \begin{align}\label{eq:conf-bound000}
       \ah(u-\uh,v)
       &=(f,v)-\ah(\uh,v)\nonumber\\
        &= (f,v-\Pi^c_hv)-\ah(\uh,v-\Pi^c_hv)  
      + \{ (f,\Pi^c_hv)-\ah(\uh,\Pi^c_hv) \}
         \nonumber\\
           &=  (f-f_h +a_h\uhz - a\uhz,\hat{v})+(R_h,\hat{v}) -(J_h, \hat{v}).
         \end{align}
          To obtain~\eqref{eq:conf-bound000}, we have added and subtracted the data $(f_h-a_h\uhz,\hat{v})$ and for the last two terms of~\eqref{eq:conf-bound000} we apply integration by parts to $\eps(\wgrad \uh,\wgrad\hat{v})$.
\end{proof} 
  \end{lemma}

       \begin{lemma}[Conforming term bound]\label{lemma-conf} Let $u$ solve~\eqref{eq:cd1}--\eqref{eq:cd2} and $\uh \in V_h$ solve~\eqref{eq:cd-wg} and express $\uh=\uh^c+ \uh^r$ according to~\eqref{eq:decomp}.  
       For any $v\ne 0\in H_0^1(\Omega)$,      
       \begin{align}\label{eq-conf}
        \normmm(u-\uh^c) + |\bb(u-\uh^c)|_* \lesssim 
        \frac{\ah(\uh^r,v)}{\normmm(v)}
          + \eta_h 
          +
          osc(f, a).
       \end{align}
     
       \end{lemma}  
       \begin{proof} Thanks to the inf-sup condition~\eqref{eq:inf-sup}, for any  $v\in H_0^1(\Omega)\setminus \{0\}$,
       \begin{align}\label{eq:conf-bound0}
      \normmm(u-\uh^c) + |\bb(u-\uh^c)|_* \lesssim \frac{\ah(u-\uh^c,v) }{\normmm(v)}.
       \end{align}
  Now, for the rhs of~\eqref{eq:conf-bound0}, using $\uh^c=\uh-\uh^r$ and~\eqref{eq:err-eqn} we have
        \begin{align}\label{eq:conf-bound1}
       \ah(u-\uh^c,v)
       &=(f,v)-\ah(\uh-\uh^r,v)\nonumber\\
        &=
           \ah(\uh^r,v)+
           (f,v)-\ah(\uh,v)   \nonumber\\
        &= \ah(\uh^r,v)+(osc(f,a),\hat{v})+(R_h, \hat{v})-\langle J_h, \hat{v}\rangle.
         \end{align}
     In view of~\eqref{eq:conf-bound1} and~\eqref{eq:conf-bound0} and using the approximation properties~\eqref{eq:inter_cell} and~\eqref{eq:inter_cell} we obtain the desired upper bound.
         \end{proof}
  We are now in the position to prove the reliability of the estimator.                                         
          \begin{theorem}[Reliability] Let $u$ solve~\eqref{eq:cd} and $\uh \in V_{h}$ be the WG approximation to~\eqref{eq:cd-wg}. Then, 
              \begin{align}
              \normmm(u-\uh)
              +
              |\bb(u-\uhz)|_*
               \lesssim \eta_h + osc(f, a).
              \end{align}
              \end{theorem}
\begin{proof}
Decompose $\uh=\uh^c+\uh^r$ where $\uh^c\in V_h^c$. By the triangle inequality and Lemma~\ref{lemma-conf}, for any nonzero $v\in H_0^1(\Omega)$,
\begin{align*}
\normmm(u-\uh) +  |\bb(u-\uhz)|_*
&\le \normmm(u-\uh^c) +  |\bb(u-\uh^c)|_*
+
\normmm(\uh^r) +  |\bb\uhz^r|_*\nonumber\\
 &\lesssim \frac{\ah(\uh^r,v)}{\normmm(v)}
                + \eta_h 
                +
                osc(f, a)
                +
                \normmm(\uh^r) +  |\bb\uhz^r|_*
\end{align*}
The result now follows thanks to Lemma~\ref{lemma-nconf}.
\end{proof}

      \begin{theorem}[Efficiency]
          \begin{align}
         \eta_h
           \lesssim   \normmm(u-\uh) + |\bb(u-\uhz)|_* + osc(f,a).
          \end{align}
          \end{theorem}
          \begin{proof}
    Efficiency will be shown through the following lemmas which employ the bubble function techniques introduced by Verf\"{u}rth\cite{verfurth2005robust}.
    We remark that the arguments presented here are simpler than its DG counterpart presented in~\cite{schotzau2009robust} thanks to the simpler weak form.
          \end{proof}
          
          \begin{lemma}[Cell efficiency]
              \begin{align}\label{eq:Cell-eff}
                   \Big(\sum_{T\in \TTh}\big(\eta_{T,1}^2 + 
                                           \eta_{T,2}^2\big)\Big)^{1/2}
                     \lesssim   \normmm(u-\uh) + |\bb(u-\uhz)|_* + osc(f,\bb,a).
                    \end{align}
          \end{lemma}
          \begin{proof}
It is easy to see that $\sum_{T\in \TTh}\eta_{T,2}^2 \le \normmm(u-\uh)^2$.
          Now for bounding $\sum_{T\in \TTh}\eta_{T,1}^2$, we let $b_T$ denote the element-bubble function described in~\citep[p. 1771]{verfurth2005robust} for any cell $T\in \TTh$.
          We first take 
          \[
          v_T=\alpha_T^2 R_h b_T.
          \]
   Since $R_h$ belongs to a finite dimensional space on $T$, a local equivalence of norms yields
   \begin{align}
   &\int_T \alpha_T^2 |R_h|^2 \ dx
   \lesssim   
   \int_T  \alpha_T^2  R_h^2 \ b_T \ dx =   \int_T  R_h \ v_T \ dx.\label{eq:eff0}
   \end{align}
   Since the exact solution $u$ satisfies the strong form $-\eps \div \nabla u + \div(\bb u) +a u =f $ on each $T \in \TTh$, by adding and subtracting the exact data and applying integration by parts to $ \int_T \div (\eps\wgrad u - \eps\wgrad \uh)  v_T \ dx $ in conjunction with $v_T|_{\partial T}=0 $, we have
    \begin{align}
    &\int_T  R_h \ v_T  \ dx
       =   \int_T \eps(\wgrad u - \wgrad \uh) \cdot \wgrad v_T \ dx \  + \nonumber\\
     &   \int_T \big(\wdiv(\bb u) -\wdiv(\bb \uh) + a(u-\uhz) \big)\nabla v_T \ dx \  + \nonumber\\
   &   \int_T (f_h-f + (a_h-a)\uhz)v_T \ dx.  \label{eq:eff1}
    \end{align}   
    Now combining~\eqref{eq:eff0} and~\eqref{eq:eff1}, summing over all $T \in \TTh$ and applying Cauchy Schwarz inequality we obtain
    \begin{align}
   \sum_{T\in \TTh}\eta_{T,1}^2
   & 
   \lesssim \big(\normmm(\eh)+ |\bb \eh|_*+ osc(f,a)\big) 
 \Big(\sum_{T\in\TTh}\normmm(v_T)^2 \  + \alpha_T^{-2}\norm{v_T}_T^2\Big)^{1/2}.
      \end{align}  
Thus, 
  \begin{align}
        & \sum_{T\in \TTh}\eta_{T,1}^2
         \lesssim
         \Big(\normmm(\eh)+ |\bb \eh|_*+  osc(f,a)\Big) \nonumber \\
            & \quad \ \Big( \normmm(v_T)^2 \  +  \sum_{T\in\TTh}\alpha_T^{-2}\norm{v_T}_T^2\Big)^{1/2}.
          \label{eq:eff2}
      \end{align} 
      Using the following property of the element bubble function
        \[
        \normmm(\sigma b_T) \lesssim
             \alpha_T^{-1}  \norm{\sigma }_{L^2(T)} , \ \norm{b_T}_{\infty}=1
        \]   
        $\normmm(v_T)^2 \lesssim \alpha_T^{2}  \norm{R_h }_{L^2(T)} $ and $\alpha_T^{-2}\norm{v_T}_T^2 \lesssim \alpha_T^{2}  \norm{R_h }_{L^2(T)}$
         the assertion follows.   
          \end{proof}
             \begin{lemma}[edge efficiency]
                        \begin{align}\label{eq:edge-eff}
                             \Big(\sum_{E\in \mathscr{E}_h}\eta_{E}^2\Big)^{1/2}
                               \lesssim   \normmm(u-\uh) + |\bb(u-\uhz)|_* + osc(f).
                              \end{align}
                    \end{lemma}
  \begin{proof}
For any edge $E\in \mathscr{E}_h$, let $b_E$ denote the edge-bubble function described in~\citep[p. 1771]{verfurth2005robust}.
          We first take 
          \[
          v_E=\eps^{-1/2}\alpha_E J_h b_E.
          \]
   Since $R_h$ belongs to a finite dimensional space on $T$, once again we rely on a local equivalence of norms to obtain
   \begin{align}
   &\int_E  J_h \ v_E  \ ds
   \lesssim   
   \int_E \jump{\eps \bn \cdot \wgrad \uh}  \ v_E  \ ds =   -\int_E \jump{\eps \bn \cdot \wgrad \eh}  \ v_E  \ ds. \label{eq:eff3}
   \end{align}
Let ${\omega_E}\subset \TTh$ denote the compact support of $E$ and $v_{\omega_E}$ denote the extension of edge bubble function to ${\omega_E}$.
    \begin{align}
    &-\int_E \jump{\eps \bn \cdot \wgrad \eh}  \ v_E  \ ds
    =
     - \int_{\omega_E} \div \eps(\wgrad u - \wgrad \uh)  v_{\omega_E} \ dx \   \nonumber\\
   &  - \     \int_{\omega_E} \eps(\wgrad u - \wgrad \uh) \cdot \nabla v_{\omega_E} \ dx \  \nonumber\\
     &  \pm\  \int_{\omega_E} \big(\wdiv(\bb u) -\wdiv(\bb \uh) + a u- a_h\uhz \big)\nabla v_{\omega_E}  \ dx. \   \nonumber\\
     &=  \int_{\omega_E} osc(f,a)v_{\omega_E} dx    
     -\  \int_{\omega_E} \big(\wdiv(\bb\eh) + a\ehz \big) v_{\omega_E}  \ dx
        \nonumber\\ 
    &+   \int_{\omega_E} R_h v_{\omega_E} \ dx   - \int_{\omega_E} \eps \wgrad \eh \cdot \nabla v_{\omega_E} \ dx    \nonumber\\
       &\le \big(osc(f,a) +|\bb \ehz|_* + \norm{a}_{\infty}\norm{\ehz}\big)\norm{v_{\omega_E}}_{\omega_E} + \sum_{T\in\omega_E}\alpha_T\eta_{T,1}\alpha_T^{-1}\norm{v_{\omega_E}}_T\nonumber\\
       & + \norm{\eps^{1/2}\wgrad \eh}_{\omega_E}\norm{\eps^{1/2}\nabla v_{\omega_E}}_{\omega_E} 
    \end{align}   

      Using the properties of the edge bubble function
      \[
      \norm{\sigma b_E}^2_{\omega_E} \lesssim
      \eps^{1/4}\alpha_E^{1/2}  \norm{\sigma }^2_{E} ,
      \]
      and
      \[
      \normmm(\sigma b_E)^2_{\omega_E} \lesssim
           \eps^{1/4}\alpha_E^{-1/2}  \norm{\sigma }^2_{E} ,
      \]   
         the assertion follows.   
          \end{proof} 
          \section{Numerical Results}\label{sec:numerical}       

       In this section we present the results of the following benchmark problems to test the performance of the estimator. The numerical implementation has been realized by using the C++ software library deal.II~\cite{dealII82,BangerthHartmannKanschat2007}.  
       Our sequence of adaptively refined rectangular meshes is constructed by selecting those elements for refinement which possess the top $25\%$ of the largest local indicators $\eta_T$. Since we are using rectangular meshes, local grid refinement inevitably leads to irregular meshes, i.e., not every edge
       of a cell is also a complete edge of its neighboring cell. Consistent with our implementation, we restrict this irregularity to one-irregular meshes, that is, any edge of a cell is shared by at most two cells on
       the other side of the edge.
       
     All the experiments were performed on the unit square $\Omega=(0,1)^2$ with the initial grid consisting of $16 \times 16 $ elements, using polynomial degree $k=2,3$. We remark that these examples have been previously investigated in~\cite{ayuso2009discontinuous,CHEN2017107}.
     
     Since the norm used in proving the reliability and efficiency of the estimator involves the energy norm~\eqref{eq:energy-norm} and the unusual norm $|\bb(u-\uhz)|_*$, for the numerical experiments, we calculate $|\bb(u-\uhz)|_*$, by using the upper bound $\eps^{-1/2}\norm{u-\uhz}$ which has been described in~\citep[Remark 3.5]{schotzau2009robust}.

        \subsection{Boundary Layer Benchmark}\label{ex_1}
                \begin{figure}[hbt!]
                \centering                                                
                          \includegraphics[width=.32658\textwidth]{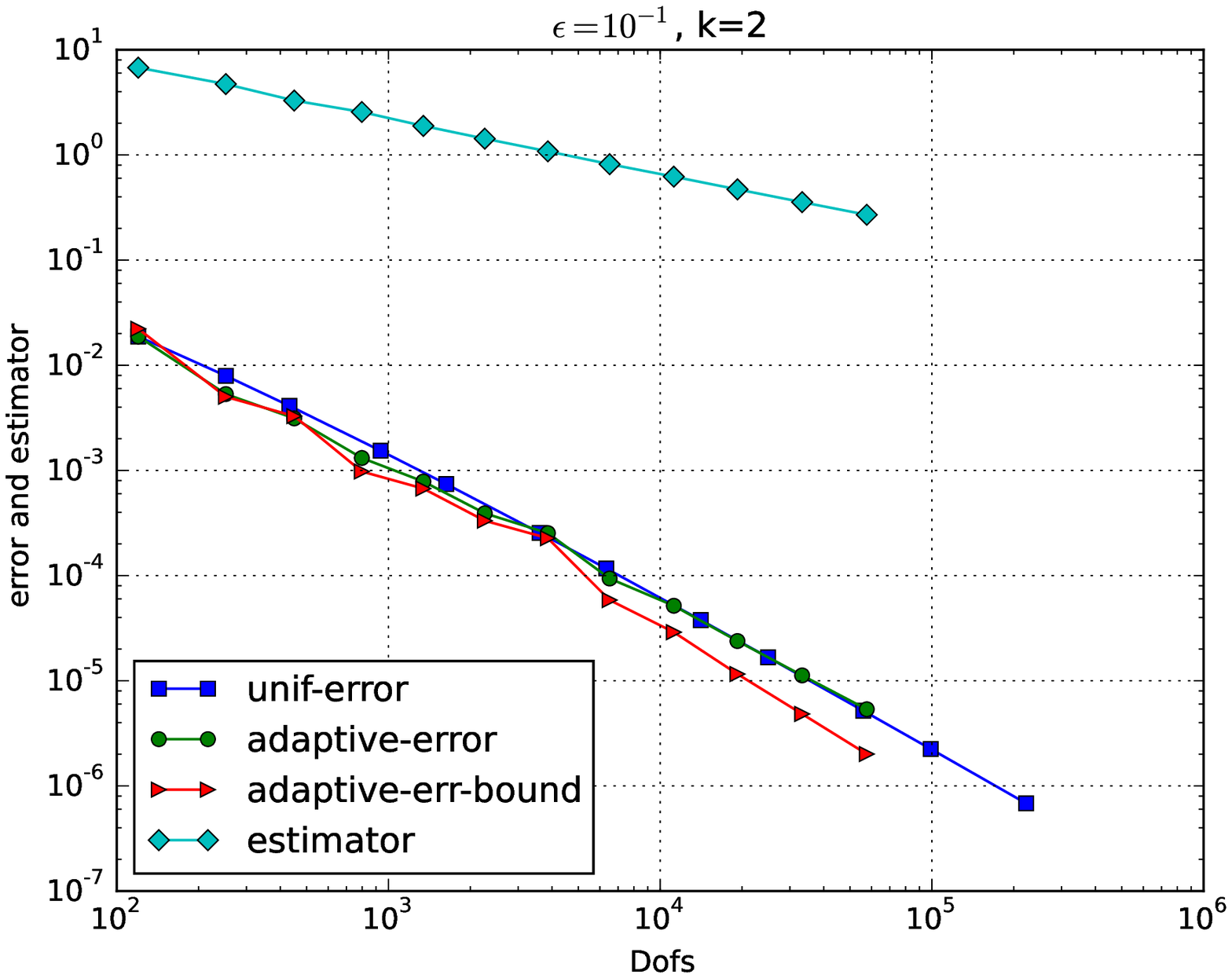}
                           \includegraphics[width=.32658\textwidth]{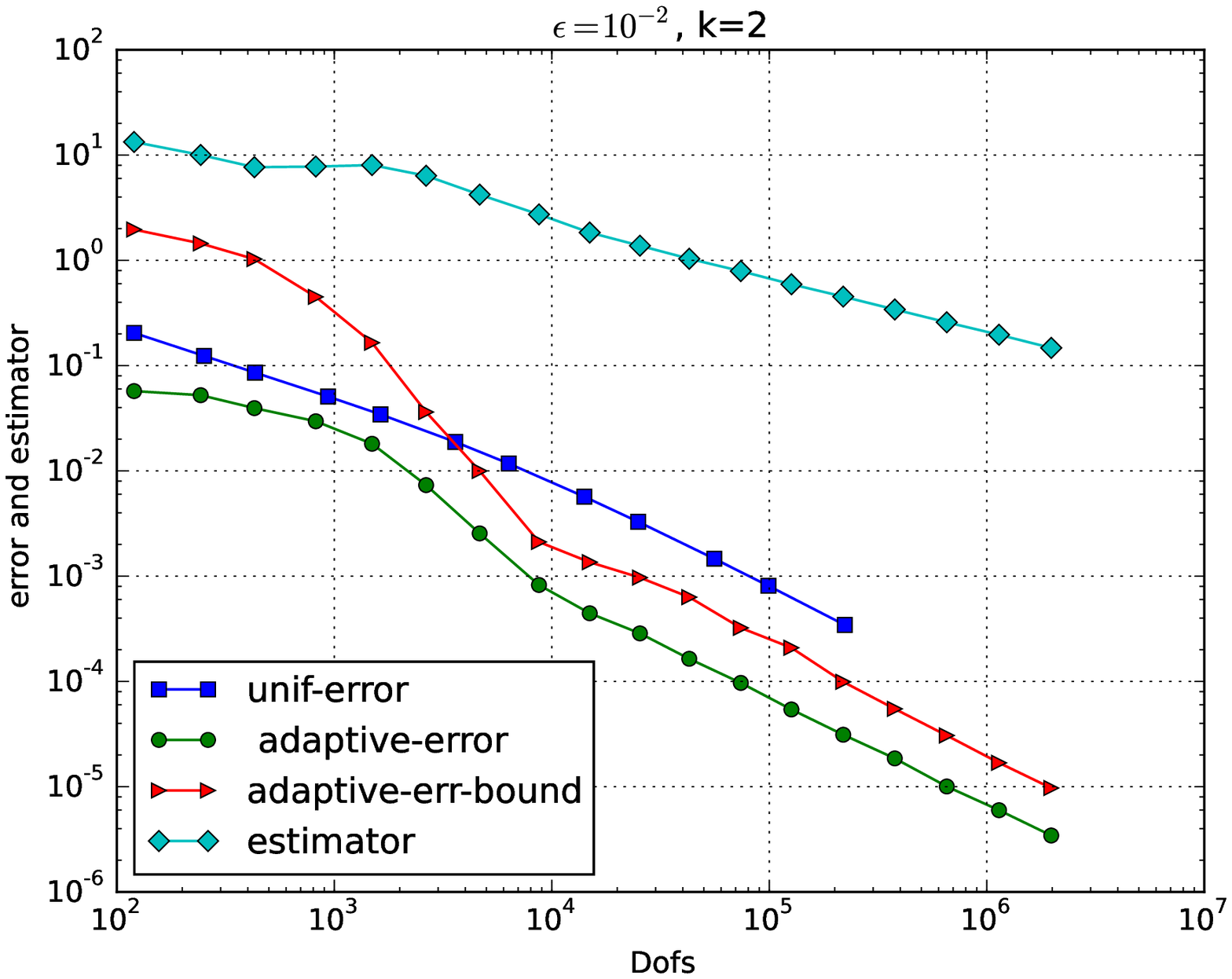} 
                           \includegraphics[width=.32658\textwidth]{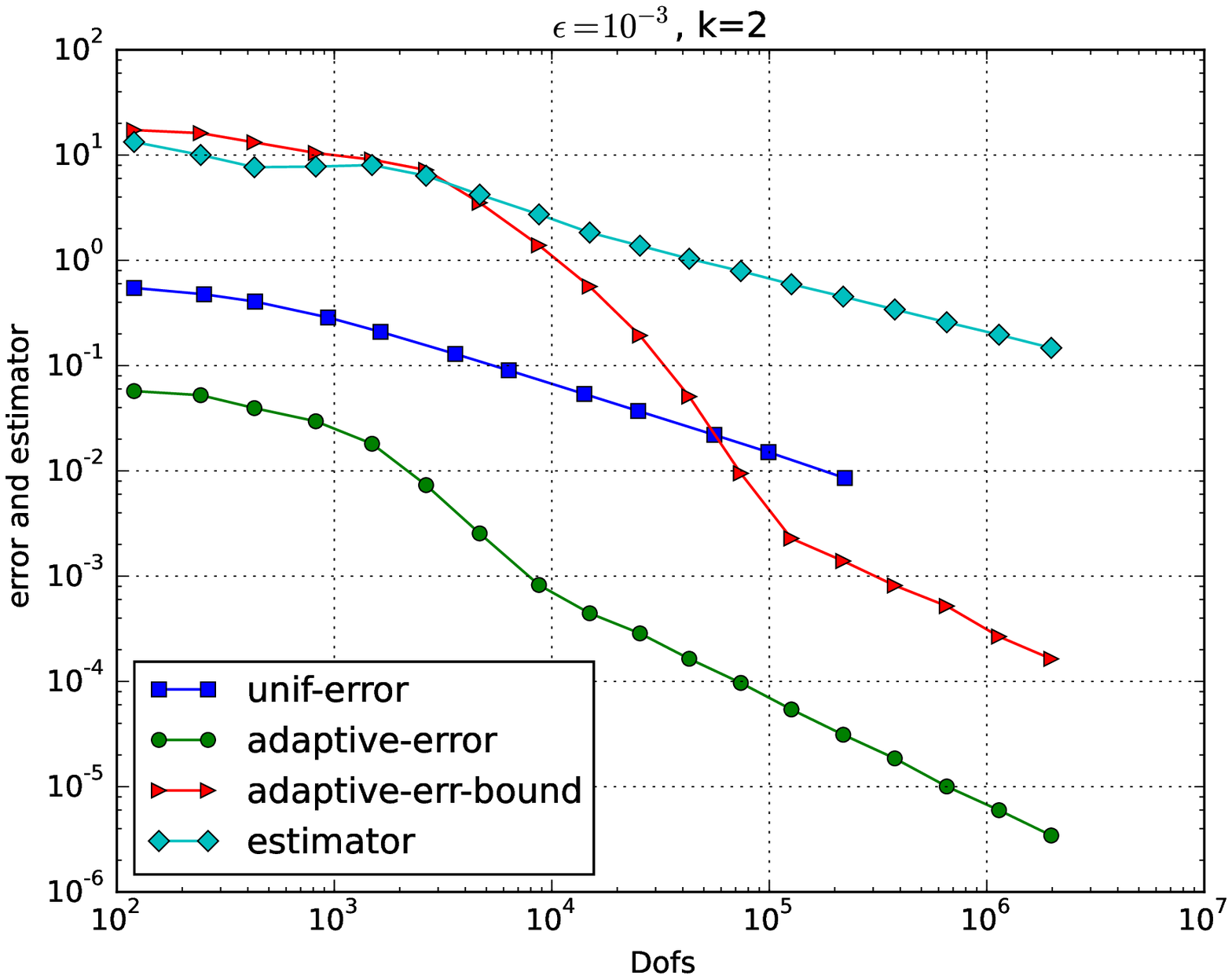}                    
           \caption{Boundary Layer Benchmark: Convergence history for $\eps=10^{-1}, \ 10^{-2}, \ 10^{-3}$ and $k=2$.}                                   
                                                          \label{fig:convergence-history-p2}
                                                         \end{figure}
                        
                                \begin{figure}[!hbt]
                             \centering            
                                 \includegraphics[width=.32\textwidth]{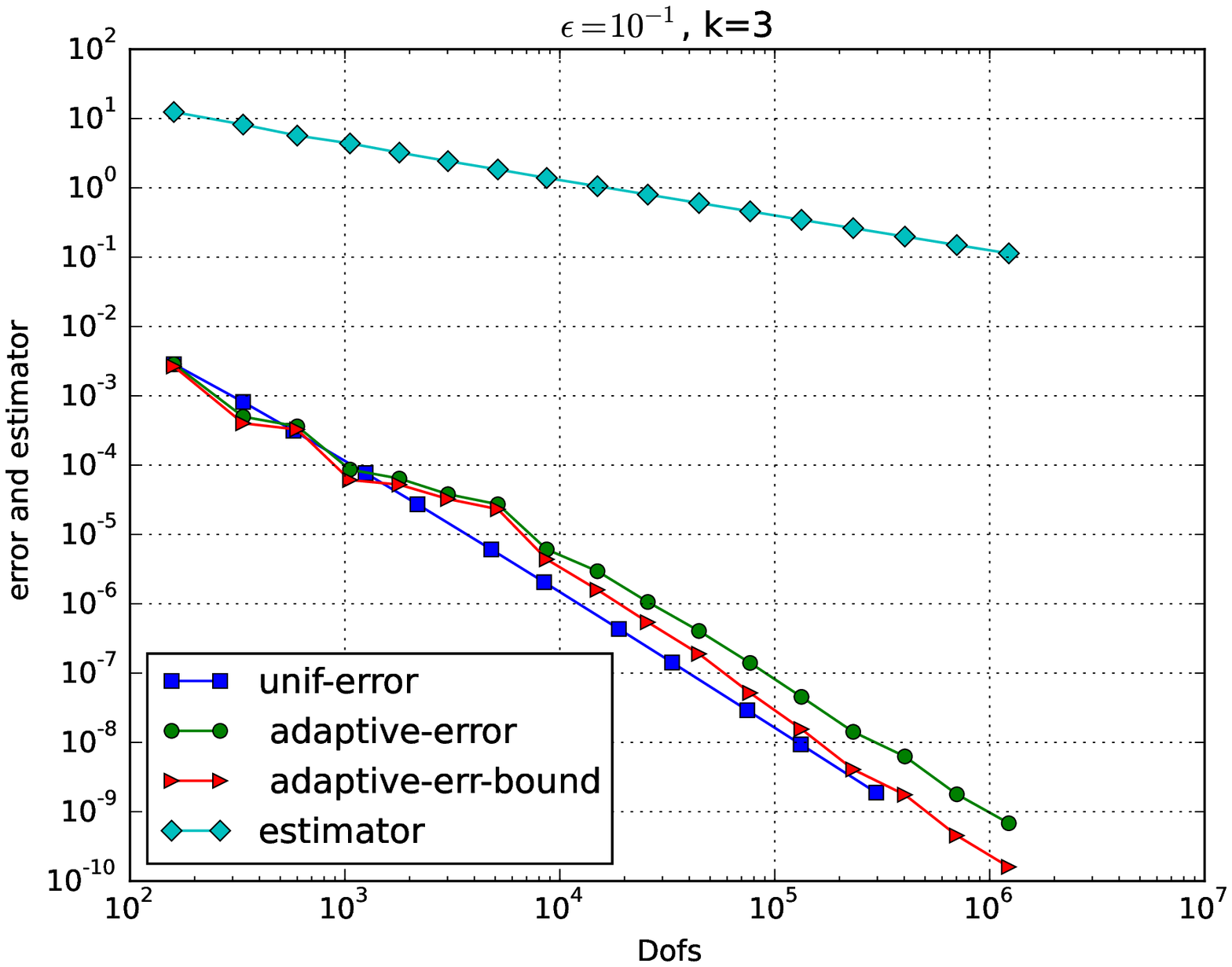}
                                    \includegraphics[width=.32\textwidth]{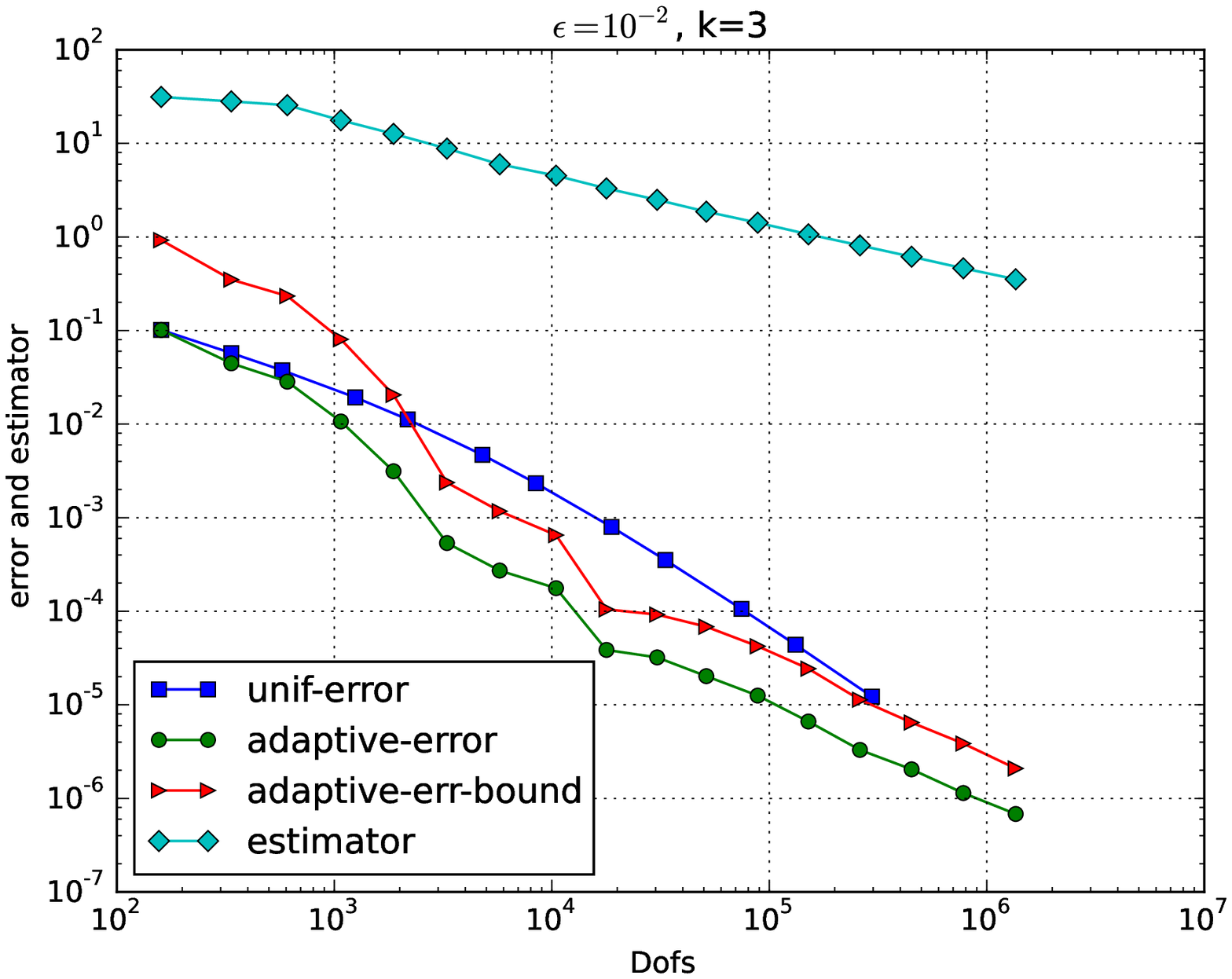} 
                            \includegraphics[width=.32\textwidth]{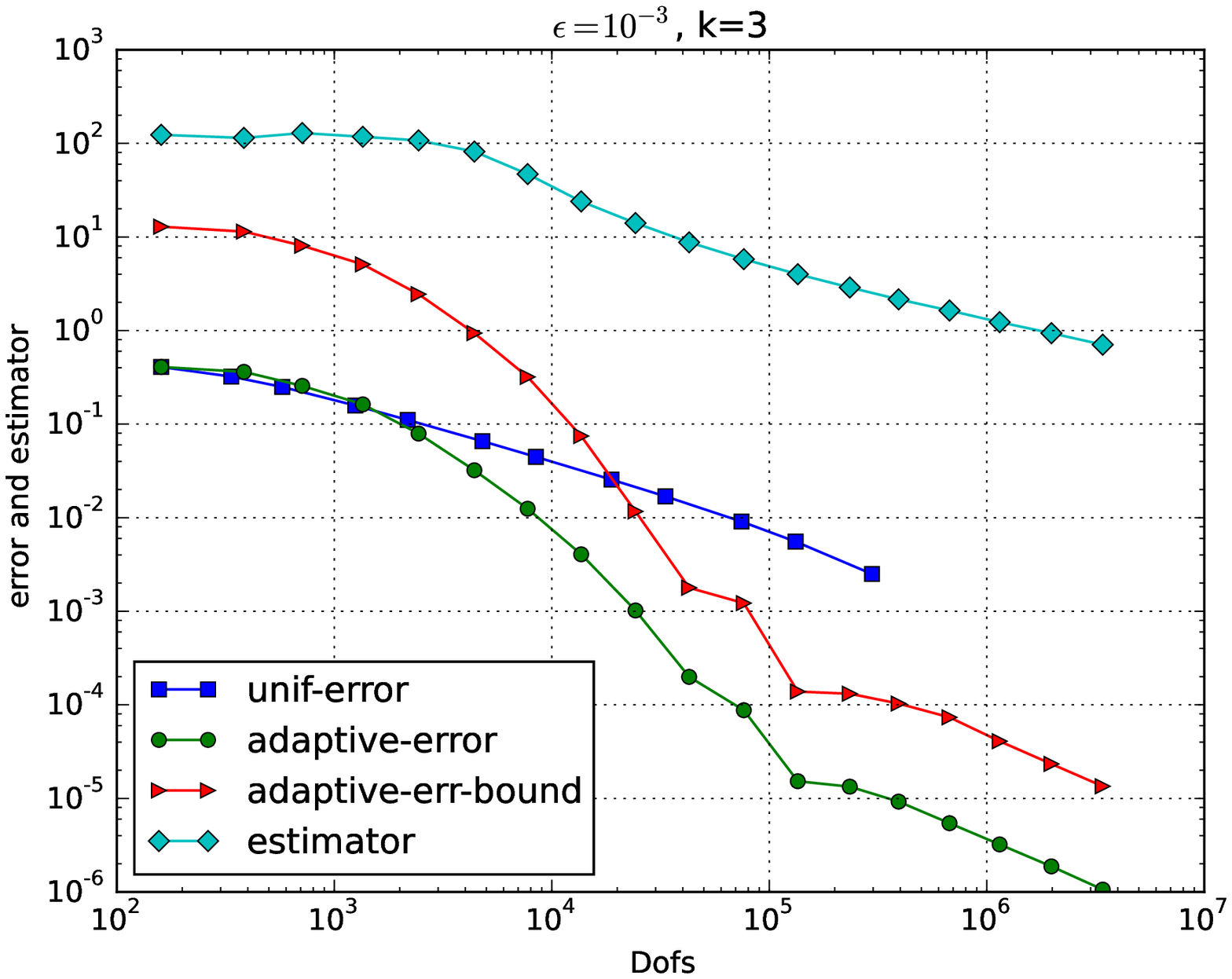}   \caption{Boundary Layer Benchmark: Convergence history for $\eps=10^{-1}, \ 10^{-2}, \ 10^{-3}$ and $k=3$.}                                                 
             
                                                            \label{fig:convergence-history-p3}
                                                          \end{figure} 
      
                                      \begin{figure}[!hbt]
                                   \centering            
                                        \includegraphics[width=.32\textwidth]{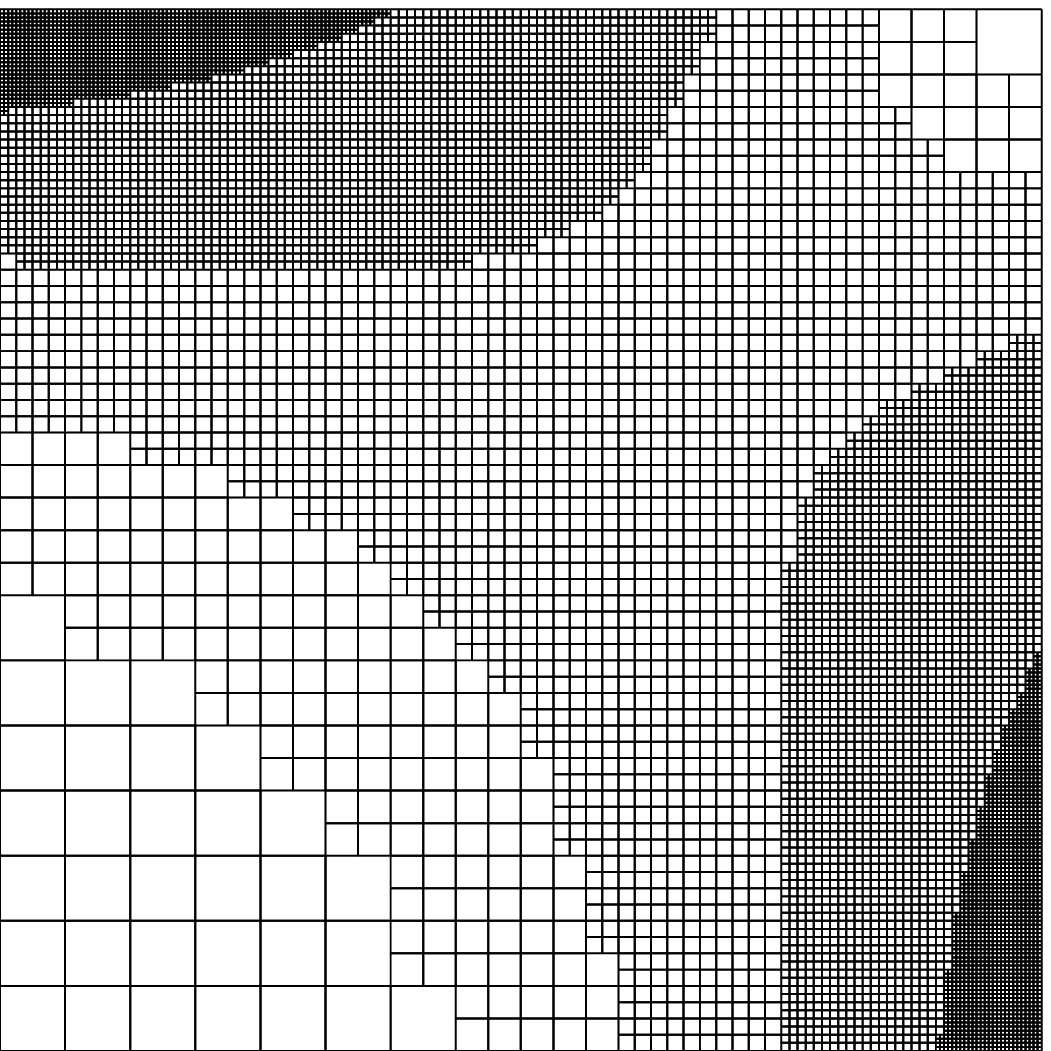} \
                          \includegraphics[width=.32\textwidth]{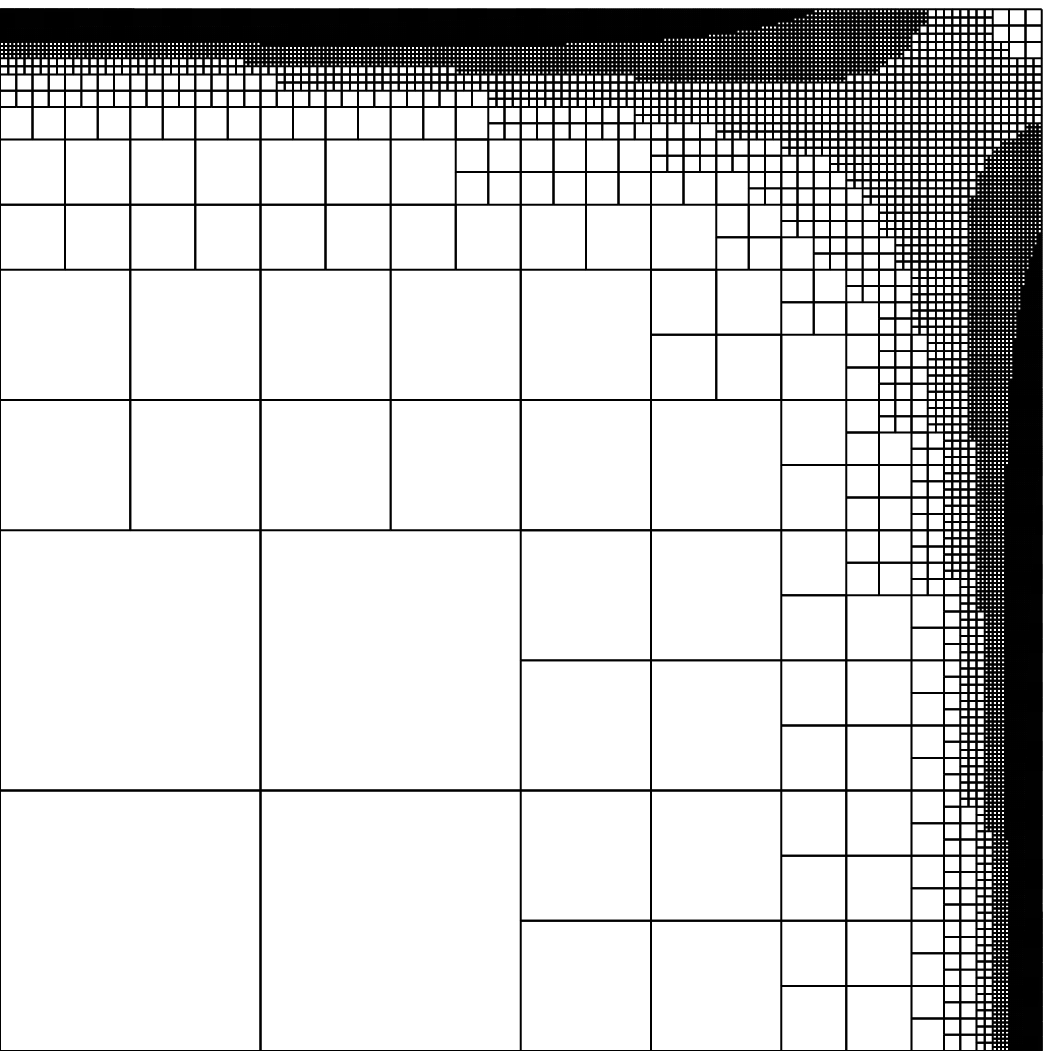} \    
                          \includegraphics[width=.32\textwidth]{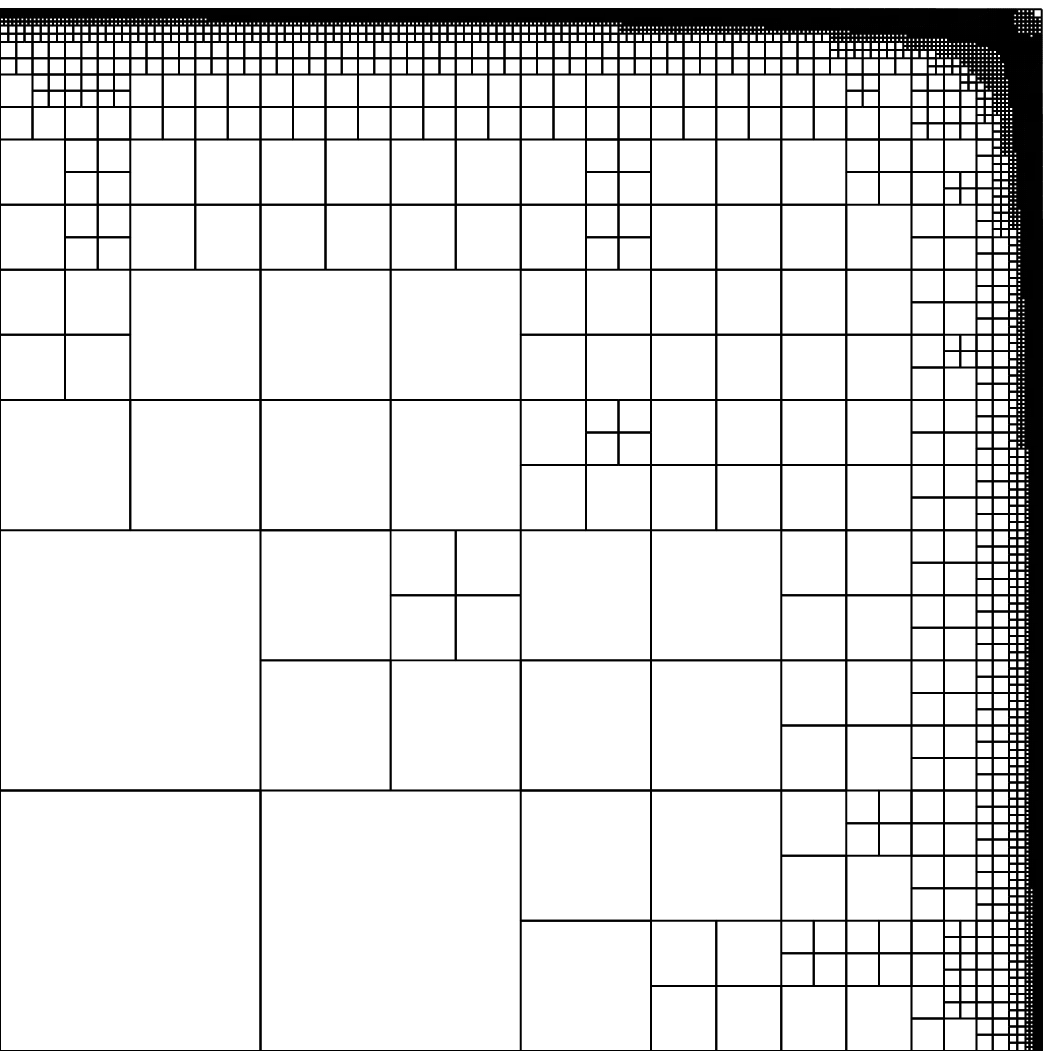}                                                  
                     \caption{Boundary Layer Benchmark: Refined meshes after 11 levels of refinement for $\eps=10^{-1},$ {\it{(left)}} $10^{-2}$, {\it{(center)}}  $10^{-3}$ {\it{(right)}} and $k=3$.} 
                                                                  \label{fig:refined-meshes-blayer}
                                                                \end{figure}                                                                                                           
       We take $\bb=(1,1)^T$ and $a=0$ with the diffusive coefficient varying between $10^{-3}$ to $10^{-1}$. We pick the boundary conditions and the right-hand side $f(x,y)$ so that the analytical solution to~\eqref{eq:cd} is 
                              \begin{align*}
                              u(x,y)= x + y\Big(1 - x  \Big)  + \frac{e^{-\eps^{-1}}- e^{{-(1-x)(1-y)\eps^{-1}}}}{1-e^{-\eps^{-1}} }.
                              \end{align*}
                              
                               It is well known that when $\eps$ is small, the solution develops boundary layers at $x=1$ and $y=1$. These layers have width $\mathcal{O}(\eps)$. 
In Figures~\ref{fig:convergence-history-p2} and~\ref{fig:convergence-history-p3}, we demonstrate the performance of our estimator for $\eps =10^{-1}, \ 10^{-2},\ 10^{-3}$ using $k=2$ and $k=3$ respectively against the total degrees of freedom. In the same figures, we also plot the ``true" energy error $\normmm(u-\uh)$ on adaptively refined meshes denoted by the curve ``adaptive-error" and the curve ``adaptive-err-bound". The latter serves as a computable upper bound for $|\bb \ehz|_*$ (see Remark~\ref{rem:bound}).
We compare the performance of ``adaptive-error" curve with the ``true" energy error computed on uniformly refined meshes depicted by the curve ``unif-error" plotted in the same figures.

As expected the estimator always overestimates the ``true" energy norm thereby confirming the reliability of the estimator. Furthermore, after the boundary layers are sufficiently resolved for each choice of $\eps$, the asymptotic regime is achieved. 
We remark that through the adaptive mesh refinement, we can achieve the expected convergence rates particularly for the intermediate regime $\eps=10^{-3}$, where the poor numerical behavior was observed on uniformly refined meshes. See~\citep[Example 2]{LinYeZhaZhu18} and~\citep[Example 2]{ayuso2009discontinuous}. This phenomena is numerically verified in~Figures~\ref{fig:convergence-history-p2} and~\ref{fig:convergence-history-p3} for the uniform error curves. We notice that for $\eps=10^{-1}$, the adaptive error and uniform error curves are almost indistinguishable for $k=2$ and certainly the uniform error curve outperforms the adaptive error curve for $k=3$. The reason for this observation is that for $\eps=10^{-1}$, it takes only a few levels of uniform refinement for the mesh size to be fine enough to resolve the layer of width $10^{-1}$.

In Figure~\ref{fig:refined-meshes-blayer}, we present the adaptively refined meshes after 11 levels of refinement for $k=3$. We notice a pronounced refinement along the lines $x=1$ and $y=1$ which suggests that our estimator accurately detects the boundary layers and is able to resolve them. We also notice that for $\eps=10^{-1}$, the width of the layer is ``large" enough to be resolved within just a few levels of refinement. 
                    \subsection{Internal Layer Benchmark}\label{ex_2}
             \begin{figure}[hbt!]
                                                                   \centering
                                                                              
                                                           \includegraphics[width=.42\textwidth]{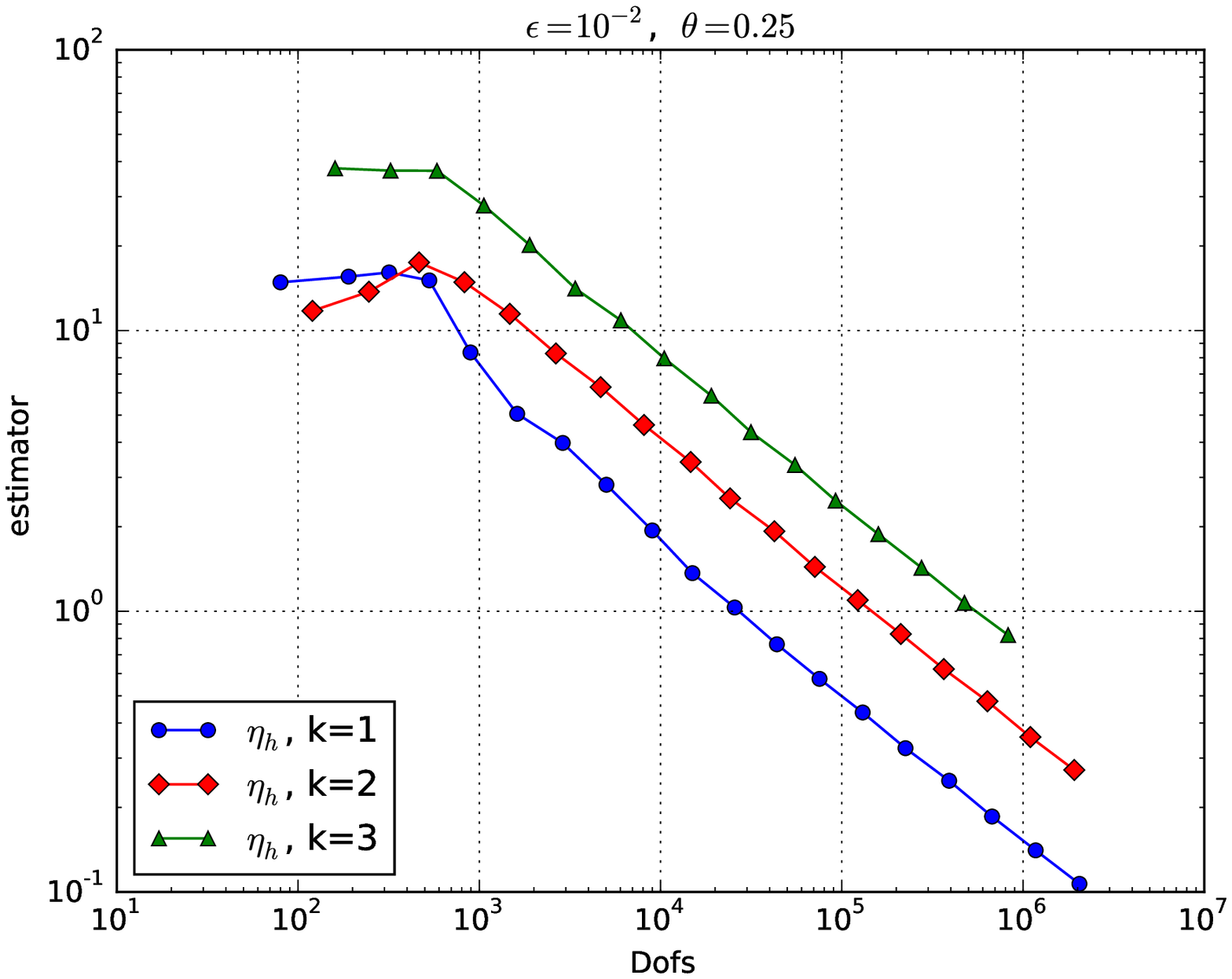}      
                                                                               \includegraphics[width=.42\textwidth]{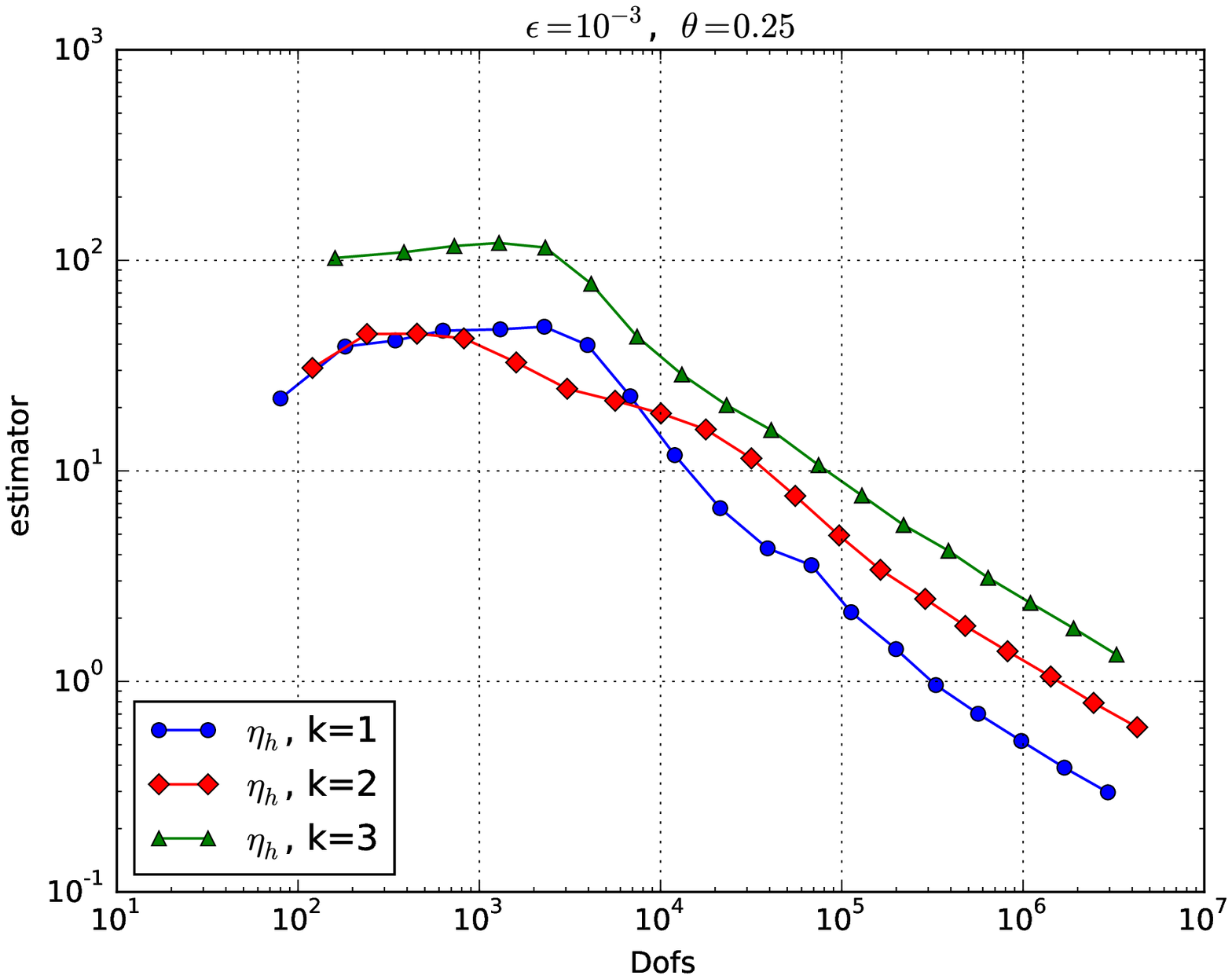}

                                            \caption{Internal Layer Benchmark:~Convergence behavior for $\eps =\ 10^{-2},$ {\it{(left)}} $\eps= 10^{-3}$ {\it{(right)}} using $ k=3 $.}\label{fig:internal-layer}
                                          \end{figure}              
                        \begin{figure}[hbt!]
                                     \centering                                        
                                     \includegraphics[width=.4\textwidth]{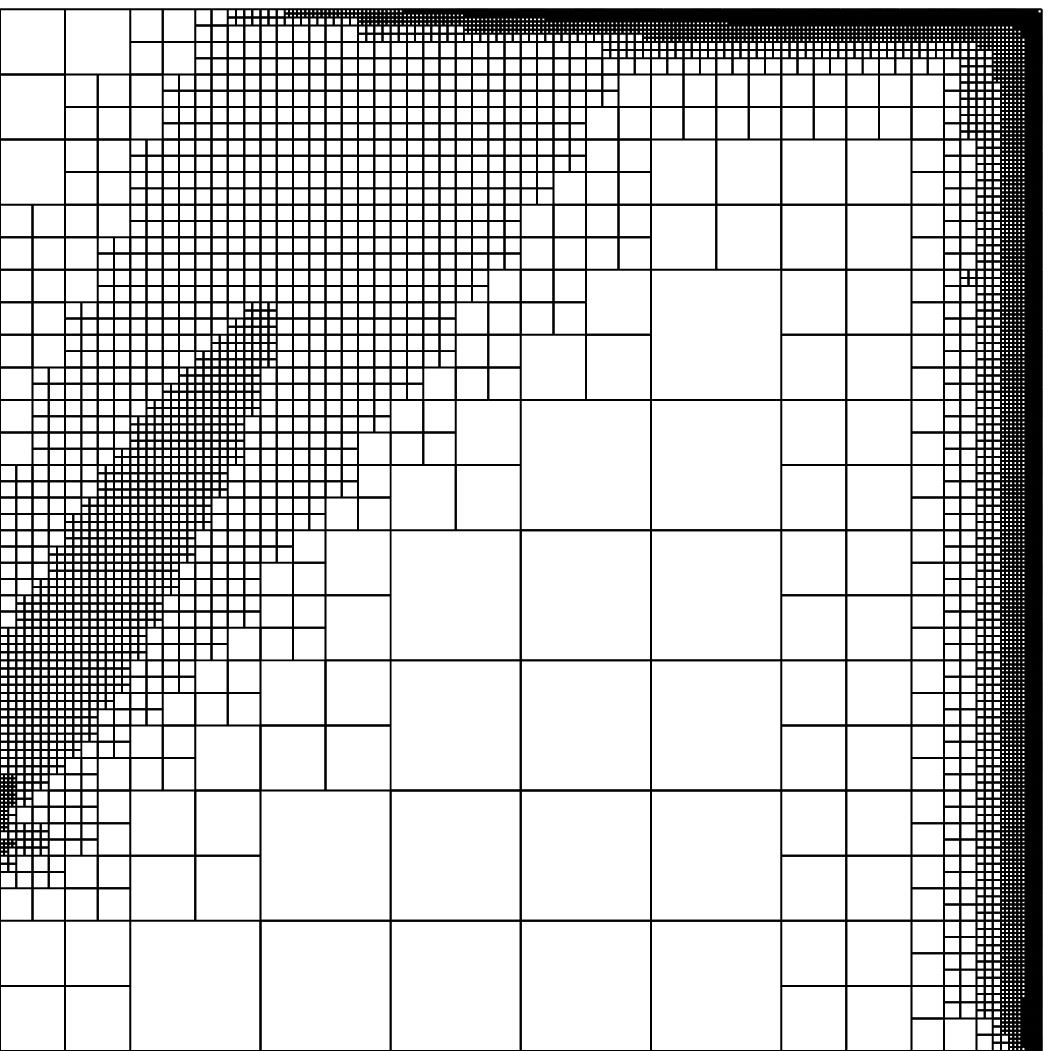} \quad
                                     \includegraphics[width=.4\textwidth]{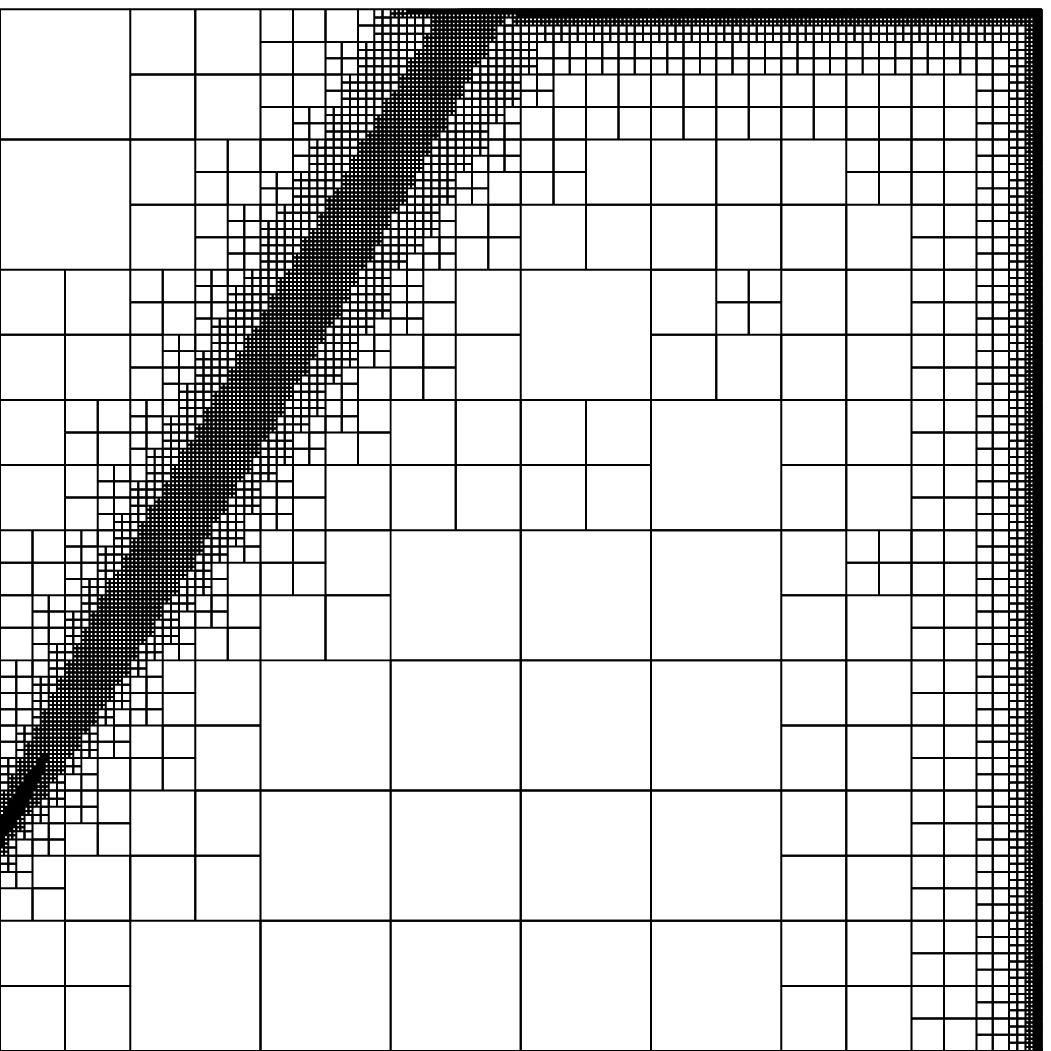} \
                  \caption{Internal Layer Benchmark:~Adaptively refined  after 11 levels of refinement for $\eps=10^{-2}$ {\it{(left)}} and $\eps=10^{-3}$ {\it{(right)}} using $ k=3 $..}
                                                      \label{fig:ilayer-mesh-refinement}
                                                          \end{figure}    
               
                    In this example, we present the performance of the estimator in the presence of internal layers.
                     We choose $\bb= (\half, \frac{\sqrt{3}}{2})$, $a=0$, $f(x,y)=0$ and the Dirichlet boundary conditions are chosen as:
                            \[
                            u(x,y)=\begin{cases}
                            1 \text{ on } [0,1]\times \{0\},\\
                            1 \text{ on } \{0\} \times [0,\frac{1}{5}],\\
                            0 \text{ otherwise. }
                            \end{cases}
                            \]
                    Since the boundary data is discontinuous at the inflow boundary, an internal layer occurs across the domain. 
                    In the absence of the exact solution to this problem, we present the convergence of the estimator for the cases $\eps=10^{-2}$ and $\eps=10^{-3}$ in Figure~\ref{fig:internal-layer}. In Figure~\ref{fig:ilayer-mesh-refinement}, we present the meshes which are adaptively refined after 11 levels of refinement for the same values of $\eps$.
                    As expected, we observe the strong refinement concentrated along the internal layer as shown in demonstrating the performance of the estimator in successfully detecting the internal layer.
                    
   \section*{Acknowledgment}
   The author would like to thank Sara Pollock for her valuable advice and fruitful discussions related to the weak Galerkin method.                  
   \section*{References}
   
   \bibliography{nss-apnum-wg}
   
   \end{document}